\documentclass[11pt]{article}
\usepackage{amsmath,  amsfonts, amsthm, latexsym}

\newtheorem{theorem}{Theorem}[section]

\newtheorem{proposition}[theorem]{Proposition}
\newtheorem{remark}[theorem]{Remark}

\newtheorem{corollary}[theorem]{Corollary}

\def\eps{\varepsilon}
 \newcommand{\ue}{u^{\eps}} 
 \newcommand{\ha}{h_{a,c}}

\newcommand{\tr}{\operatorname{\text{tr}}}
 \newcommand{\Z}{{\mathbb Z}}
   \newcommand{\Q}{{\mathbb Q}}
  \newcommand{\N}{{\mathbb N}}
 
\newcommand{\R}{{\mathbb R}}

\begin{document}
\title{Homogenization of a semilinear heat equation}
\author{Annalisa Cesaroni\footnote{Dipartimento di Scienze Statistiche, 
Universit\`a di Padova, Padova, Italy, 
e-mail: annalisa.cesaroni@unipd.it},
Nicolas Dirr\footnote{Cardiff School of Mathematics, Cardiff University, Cardiff, UK,
e-mail: dirrnp@cardiff.ac.uk.}, 
Matteo Novaga\footnote{Dipartimento di Matematica, 
Universit\`a di Pisa, Pisa, Italy, 
e-mail: matteo.novaga@unipi.it} }

\date{}

\maketitle

\begin{abstract}
We consider the homogenization of a semilinear heat equation with vanishing viscosity and with oscillating positive potential depending on $u/\eps$. 
According to the rate between the frequency of oscillations in the potential and the vanishing factor in the viscosity, we obtain 
different regimes in the limit  evolution and we discuss the locally uniform convergence of the solutions to the effective  problem.  The interesting feature of the model is that 
 in  the strong diffusion regime the effective operator is discontinuous in the gradient entry. We get a complete characterization of the limit solution in dimension $n=1$, whereas in 
dimension  $n>1$ we discuss the main properties  of the solutions to the effective problem selected at the limit  and we prove uniqueness for some classes of initial data.    \end{abstract}

\tableofcontents

\section{Introduction} 
 We consider the following  problem: \begin{equation}\label{eqintro}  
\begin{cases} u^\eps_t    - \eps^\alpha \Delta \ue- g\left(\frac{\ue}{\eps } \right)=0 & \text{in }  \R^n\times (0, +\infty) \\  u^\eps(x,0)=u_0(x) & \text{in }  \R^n  \end{cases} \end{equation}  
where $\alpha\geq 0$ and the potential $g$ is a periodic, Lipschitz continuous and positive function.
This is a simple model for the motion of an interface in a heterogeneous medium, modeled  by $g$. These kind of equations arise
e.g. in the study of the propagation of flame fronts in a solid medium having horizontal periodic striations, see the appendix of \cite{nr} for  a survey of the physical
background motivating this equation (see also \cite{cna,m}). 

In this paper we show that, depending on the value of $\alpha$, different regimes arise in the limit evolution. 
If $\alpha=1$, then   $u^\eps$ converges locally uniformly to the unique Lipschitz 
continuous viscosity solution to 
\begin{equation}\label{limitintro}
\begin{cases} 
u_t-  \bar c(|\nabla u|) =0\quad \text{in }  \R^n\times (0, +\infty)
\\
u(x,0)=u_0(x),\end{cases}
\end{equation}
where   $\bar c:[0, +\infty)\to [0, +\infty)$  is a continuous, nondecreasing and nonnegative function, 
which  satisfies  \[\bar c(0)=  \left(\int_0^1 \frac{1}{g(s)}ds\right)^{-1} \qquad \text{and}  \qquad \lim_{|p|\to +\infty} \bar c(|p|)
=\int_0^1 g(s)ds.\] 
In particular $$\left(\int_0^1 g(s)^{-1}ds\right)^{-1}\leq \bar c (|p|)\leq \int_0^1 g(s)ds,$$ and the second inequality is strict if $g$ is nonconstant. 
 Since the solution $u$ to \eqref{limitintro} is Lipschitz in $x$, with the same Lipschitz constant of the initial datum, 
necessarily the average speed is  less than $\bar c(\|\nabla u_0\|_\infty)$. 

In the case $ \alpha>1$  the limit problem is very simple and reads 
\begin{equation}\begin{cases} 
u_t= \left(\int_0^1 g(s)^{-1}ds\right)^{-1}\quad \text{in }  \R^n\times (0, +\infty)
\\
u(x,0)=u_0(x).\end{cases}\end{equation}
In particular the solutions to \eqref{eqintro} converge locally uniformly to
$u_0(x)+t \left(\int_0^1 g(s)^{-1}ds\right)^{-1}$. 

In the case $0<\alpha<1$, the limit problem is 
\begin{equation}\label{limit01}
\begin{cases} 
u_t-  \bar c_{-}(|\nabla u|) =0\quad \text{in }  \R^n\times (0, +\infty)\\
u(x,0)=u_0(x),\end{cases}\end{equation}
where \[\bar c_-(|p|)= \begin{cases}  \int_0^1 g(s)ds & p\neq 0\\  \left(\int_0^1  g(s)^{-1}ds\right)^{-1} & p=0.\end{cases}\]
In   the limiting case $\alpha=0$, the limit problem is given by 
\begin{equation}\label{limit0}
\begin{cases}  u_t- \bar F(\nabla u, \nabla^2 u) =0\quad \text{in }  \R^n\times (0, +\infty)\\
u(x,0)=u_0(x) \end{cases}\end{equation}  where  \[  \bar F(p, X))= \begin{cases}  \tr X+\int_0^1 g(s)ds & p\neq 0\\  
\min\left(\tr X+\int_0^1 g(s)ds, \left(\int_0^1  g(s)^{-1}ds\right)^{-1}\right) & p=0.\end{cases}\] 
The functions   $\bar c_- $ and $\bar F$ are both   {\it discontinuous} functions,  such phenomenon is unusual in homogenization problems, and 
makes the analysis of this limit more challenging.

Due to the lack of uniqueness of  solutions to Hamilton-Jacobi equations with discontinuous Hamiltonian, in this case we prove that along subsequences the solution $u^\eps_\alpha$ to \eqref{eqintro} converges locally uniformly to a viscosity solution of the limit problem.  We also provide a quite detailed description of which are the solutions of the discontinuous problem selected 
in the limit, and  we identify the asymptotic  speed  of propagation at strict  maxima, at strict minima and at saddle points (with respect to $x$) of the limit function. This result allows us to obtain a complete description of the limit function for some classes of  initial data.  In particular, if the initial data is either monotone in one direction or convex, we prove that the solutions to \eqref{eqintro} converge locally uniformly to
$u_0(x)+t \int_0^1 g(s)ds$ when $\alpha\in (0,1)$, and to the  solution to $u_t-\Delta u= \int_0^1 g(s)ds$ with initial datum $u_0$ when $\alpha=0$. If, on the other hand, the initial data is  a radially simmetric function, which has a unique maximum point, then the limit function is given, for $\alpha\in (0,1)$,  by $\min\left(u_0(x)+t\int_0^1 g(s)ds, \ \max u_0+ t\left(\int_0^1 g(s)^{-1}ds\right)^{-1}\right)$.  As a consequence we show further properties of the limit function $u$, 
when the initial datum is bounded from above. 

In particular, in the one-dimensional case, we are able to prove the full convergence of the solutions 
$u^\eps_\alpha$ and the uniqueness of the limit function $u$ for $\alpha\in (0,1)$, see Theorem \ref{teounicita}. 

\smallskip

Our homogenization results are based on maximum principle type arguments. In particular we provide the effective limit problem through the solution of the so-called cell problem, and then we prove the convergence of solutions  by a suitable adaptation of the perturbed test function method proposed by Evans.  
The cell problem in our case reduces to an ordinary differential equation, see \eqref{cella}, and  is obtained by making a formal asymptotic expansion in $\eps$  of the solutions to \eqref{eq}.   It  permits to define the limit differential operator and to introduce the so-called correctors, which play the role of local barriers for the evolution.  

Throughout the paper we shall assume that the potential $g$ is strictly positive, nevertheless we expect that similar results hold also in the case of a function $g$ which possibly changes sign and satisfies $\int_0^1 g(s)ds>0$. In this case, though, the analysis of the cell problem is much more involved.  In the limiting case, i.e. when  $\int_0^1 g(s)ds=0$, the cell problem has been studied in \cite{j} and \cite{m}, 
for $\alpha\ge 1$, and  it is possible to prove that the solution  $u^\eps(x,t)$  of \eqref{eqintro} converges locally uniformly to the initial datum $u_0(x)$.
More precisely, in \cite{j}  
the  following long time rescaling of \eqref{eqintro} has been considered:
\begin{equation}\label{jerr}u^\eps_t    - \Delta \ue+ \frac{1}{\eps^\alpha} g\left(\frac{\ue}{\eps } \right)=0, \end{equation} 
showing that  $u^\eps$ converge locally uniformly to a solution of a quasilinear parabolic equation, 
see \eqref{jerrlimit}, which for $\alpha>1$ is the level set mean curvature equation. 
In \cite{m}, the $1$-dimensional case has been considered for $\alpha=1$, in a more general setting.

Homogenization of periodic structures has been studied by viscosity solution methods in a long series of papers, we just recall \cite{cm, cls, cnv, ls} and references therein. However, only few papers deal  with homogenization of equations depending (periodically) on $u/\eps$, as in our case, besides the already cited  works \cite{j,m}. For first order Hamilton-Jacobi equations we recall \cite{im, bar}. Eventually, in \cite{ibm, p}  the homogenization of ordinary differential equations 
such as $u_\eps'(t)=g\left(\frac{t}{\eps},\frac{u_\eps}{\eps}\right)$ have been studied, using  respectively   viscosity solutions and $G$-convergence methods.
 
One of the main step to solve the homogenization problem is the identification of the limit operator, as we already noted. This is done by solving a suitable defined cell problem, or equivalently, by looking to periodic pulsating wave solutions to the equation \eqref{eqintro}, at the microscopic scale. 
Pulsating wave solutions with (average) slope $p\in\R^n$ are solutions to \eqref{eqintro} with $\eps=1$ of the form $\phi(x,t)-c(p)t$, where $\phi(x,t)-p\cdot x$ is a space-time periodic function and  $c(p)$ is the (average) speed of the solution. 
Notice that, since $g$ depends only on $u$, these pulsating waves are in fact traveling waves
which moves horizontally in the $p$-direction.
Such solutions are related to the correctors used in homogenization problems and are very important in the analysis of long time behavior of the  solutions to  \eqref{eqintro}, with $\eps=1$, since typically they are the long time attractors of such solutions, see for instance \cite{cn, cdn, dky, dy, cna, nr}. 
In particular in \cite{nr}, it is proved  the existence of horizontal (e.g. with slope $p=0$)  pulsating wave solutions to $u_t=\Delta u+ g(x,u, \nabla u)$, where $g$ is a positive function, which is periodic in $x,u$. The same argument  also applies to get existence of pulsating wave solutions for rational slopes $p\in\Q^n$. In \cite{cna}  a similar problem has been studied in the plane, that is existence for any slope $p$ of pulsating wave solutions (which are traveling horizontally) to 
\begin{equation}\label{eqcna}
u_t=\delta u_{xx} + g(u)\sqrt{1+u_x^2}, 
\end{equation}
with $g$ strictly positive. The authors also provide a complete description of the asymptotic speed of propagation $c(p)$, showing that it is increasing with respect to $|p|$ (as in our case) and looking also at the limit behavior as the viscosity is vanishing, that is $\delta\to 0$.
Eventually, in \cite{lou,lou2} a geometric variant of \eqref{eqcna} has been considered, for which the author is able to construct planar and V-shaped pulsating waves.

\smallskip

The paper is organized as follows. In Section \ref{notation}  we recall some notation used in the paper,
including the definition of viscosity solution. In Section \ref{sec1}, we introduce the problem and the assumptions and provide a priori estimates on solutions to \eqref{eqintro} and on their uniform limits. Section \ref{seccell} is devoted to the solution to the cell problem, in the case $\alpha=1$ and then in the case $\alpha\neq 1$, and on the analysis of qualitative properties of the limit operators. In Section \ref{secconv} we prove the main results, that is the homogenization limits. Eventually in Section \ref{secopen} we  discuss some open problems, which in our opinion could be interesting to investigate.  

\smallskip

\paragraph{Acknowledgements.}
The first author was partially supported by the GNAMPA Project 2016 "Fenomeni asintotici e omogeneizzazione" and by the research project of the University of Padova "Mean-Field Games and Nonlinear PDEs". 
The second author was partially supported by the University of Padova, the Leverhulme Trust RPG-2013-261 and EPSRC EP/M028607/1.
The third author  was partially supported by the University of Pisa via grant PRA-2015-0017.
The second and third authors wish to thank the University of Padova for the kind hospitality during the preparation of this work.

\section{Notation and preliminary definitions} \label{notation}
Given $z\in \R$, we will denote with $[z]$ the smallest integer bigger than $z$: \[[z]\in \Z,\qquad [z]-1<z\leq [z].\]

Given a smooth function $u(x,t):\R^n\times (0, +\infty)\to \R$, we will denote with $u_t $ the partial derivative with respect to $t$, with $\nabla u $, $\nabla^2 u $, $\Delta u $ resp. the gradient, the Hessian and the Laplacian of $u$ with respect to $x$. 

Given a continuous function $u:\R^n\times (0, +\infty)\to \R$, we recall the definition of the sub and superjets of $u$ at a point $(x_0,t_0)\in \R^n\times (0, +\infty)$ (see \cite{bcd}, \cite{gigabook}): 
\[D^+ u(x_0, t_0):=\{(\nabla \phi(x_0, t_0), \nabla^2 \phi(x_0,t_0), \phi_t(x_0,t_0)) :\  \phi\in C^2, \phi\geq u, \phi(x_0,t_0)=u(x_0,t_0) \},\]
\[D^- u(x_0, t_0):=\{(\nabla \phi(x_0, t_0), \nabla^2_x \phi(x_0,t_0), \phi_t(x_0,t_0)) :\  \phi\in C^2, \phi\leq u, \phi(x_0,t_0)=u(x_0,t_0) \}.\] 

We recall the definition of viscosity solution for a parabolic  system 
\begin{equation}\label{visco} \begin{cases} u_t-F(\nabla u, \nabla^2 u)=0 & \text{in }\R^n\times (0, +\infty)\\ u(x,0)=u_0(x)& \text{in }\R^n, \end{cases} \end{equation} where the differential operator $F$ is possibly discontinuous (see \cite{gigabook}).  Given a continuous function  $u:\R^n\times [0, +\infty)\to \R$, then \begin{description} 
\item 
$u$ is a subsolution   to \eqref{visco} 
if $u(x,0)\leq u_0(x)$ and $\lambda-F^\star(p, X)\leq 0$, for every $(x_0, t_0)\in \R^n\times (0, +\infty)$ and $(p, X, \lambda)\in D^+ u(x_0, t_0)$,
\item $u$ is a supersolution   to \eqref{visco} 
if $u(x,0)\geq u_0(x)$ and $\lambda-F_\star(p, X)\geq 0$, for every $(x_0, t_0)\in \R^n\times (0, +\infty)$ and $(p, X, \lambda)\in D^- u(x_0, t_0)$,\end{description} 
where $F^\star$ and $F_\star$ denote respectively the upper and lower semicontinuous envelopes of $F$. 


\section{Assumptions and basic estimates} \label{sec1} 
We assume the following conditions on the forcing term  $g:\R \to \R$:
\begin{equation}\label{g}\text{  $g$ is Lipschitz continuous   ,  $\Z$ periodic, and $g(y)> 0$ for every $y$.}\end{equation}

We consider the following Cauchy problem
 \begin{equation}\label{eq}  
\begin{cases} u^\eps_t    - \eps^\alpha\Delta \ue- g\left(\frac{\ue}{\eps } \right)=0 & \text{in }  \R^n\times (0, +\infty) \\  u^\eps(x,0)=u_0(x) & \text{in }  \R^n  \end{cases} \end{equation}  
where $\alpha\geq 0$ and  \begin{equation}\label{id}\text{  $u_0$ is a Lipschitz continuous function, with Lipschitz constant  $L$.}\end{equation} We can assume without loss of generality that  $L\in \N$. 
In the case $\alpha=0$, we make the additional assumption that $u_0\in C^{1,1}$. 

\begin{proposition} \label{seq} 
Assume \eqref{g} and \eqref{id} and let $\alpha\geq 0$. Then  \eqref{eq} admits a unique solution $u^\eps_\alpha\in C^{2+\gamma, 1+\gamma/2}$ for all $\gamma\in (0,1)$. 
Moreover, up to a subsequence, 
\[u^\eps_\alpha\to u \qquad \text{locally  uniformly in $\R^n\times [0,+\infty)$}.\] 
 For $\alpha>0$, every limit function  $u$ is   a  Lipschitz continuous function, which satisfies 
\begin{equation}\label{estlip} |u(x,t)-u(y,s)|\leq L|x-y|+\|g\|_\infty |t-s| \qquad \forall x,y\in \R^n, \ t,s\geq 0.\end{equation}
For $\alpha=0$, under the additional assumption that $u_0\in C^{1,1}(\R^n)$,  every limit function $u$ is   a  Lipschitz continuous function, which satisfies 
\begin{equation}\label{estlip2} |u(x,t)-u(y,s)|\leq L|x-y|+(\|g\|_\infty+\|\nabla^2 u_0\|_\infty) |t-s| \qquad \forall x,y\in \R^n, \ t,s\geq 0.\end{equation}

Finally, if there exist $\eta\in\R^n$, with $|\eta|=1$, and $\delta>0$ such that 
\begin{equation} \label{mono} 
\nabla u_0(x)\cdot\eta\ge\delta \qquad 
\text{for a.e. }x\in\R^n,
\end{equation} 
then  $\nabla u(x,t)\cdot\eta\ge \delta$  for a.e. $(x,t)$. 
\end{proposition}
\begin{proof} 
Due  to the Lipschitz regularity of $g$, a standard comparison principle among sub and supersolutions to \eqref{eq} holds (see \cite{pw}). So, existence and uniqueness of solutions to \eqref{eq} follow easily, and the regularity comes from standard elliptic regularity theory (see \cite{pw}).
 
Assume now that the initial datum $u_0$ has bounded Hessian. Indeed it is not restrictive, since we can  uniformly approximate the initial datum  with   a  sequence  of  smooth  functions with  bounded  Hessian. The comparison principle implies that the associated sequence of solutions converges locally uniformly to the solution to \eqref{eq} in $\R^n\times [0, +\infty)$.

Let $C=\|\nabla^2 u_0\|_\infty$. Then for every $\eps$, the functions $u_0(x)\pm (\|g\|_\infty +\eps^\alpha C)t$ are respectively super and subsolution to \eqref{eq}, which implies by the comparison principle that 
$$|u^\eps_\alpha(x,t)-u^\eps_\alpha(x,0)|\leq (\|g\|_\infty +C\eps^\alpha )t.$$ 
Hence, again applying the comparison principle, we get that for every $t,s\geq 0$ 
\begin{equation}\label{s2} |u^\eps_\alpha(x, t+s)-u^\eps_\alpha(x,t)|\leq \sup_x |u^\eps_\alpha(x,s)-u^\eps_\alpha(x,0)|\leq(\|g\|_\infty +C\eps^\alpha)s.\end{equation} 
This implies that $u^\eps$ are equi-lipschitz in $t$. 

We prove now uniform equi-continuity in $x$.
Let us consider the functions 
$$w_\pm^\eps(x,t):=u^\eps_\alpha(x+z,t) \pm\left[\frac{ L|z|}{\eps}\right]\eps.$$  
Notice that  $w^\eps_\pm$ are both  solutions to the equation in \eqref{eq}, due to the periodicity of $g$. Moreover, we have $w^\eps_+(x,0) =u_0(x+z)+\left[\frac{ L|z|}{\eps}\right]\eps\geq u_0(x+z)+L|z|\geq u_0(x)$ and $w^\eps_-(x,0)=u_0(x+z)-\left[\frac{ L|z|}{\eps}\right]\eps\leq u_0(x+z)-L|z|\leq u_0(x)$. 
By comparison principle this implies that 
\begin{equation}\label{s1}  |u^\eps_\alpha(x +z, t) -u^\eps_\alpha(x,t)| \leq \left[\frac{ L|z|}{\eps}\right]\eps\leq L|z|+\eps \qquad \forall x,z\in\R^n, \ t\geq 0.
\end{equation}
In particular, if we take $z=\eps \bar z$ where $\bar z\in \Z^n$, then \eqref{s1} gives
\[ \frac{|u^\eps_\alpha(x +\eps \bar z, t) -u^\eps_\alpha(x,t)| }{\eps |\bar z|}\leq L \qquad \forall x\in\R^n, \bar z\in \Z^n,\ t \geq 0.\]

For every $\eps>0$ we consider a Lipschitz continuous  function  $\tilde u^\eps_\alpha$, which satisfies \eqref{s2}, $ |\tilde u^\eps_\alpha(x,t)- \tilde u^\eps_\alpha(y,t)|\leq L|x-y|$ for every $x,y\in\R^n$ and $t\geq 0$,and  such that $u^\eps_\alpha\equiv \tilde u^\eps_\alpha$ on  the lattice $\eps\Z^n\times (0, +\infty)$.  
This implies that $\|u^\eps_\alpha-\tilde u^\eps_\alpha\|_\infty\leq K\eps$. Indeed, let $x\in \R^n$ and $t\geq 0$. Fix $y_\eps\in \eps\Z^n$ such that $|x-y_\eps|\leq \eps$. Then, using \eqref{s1} and the definition of $\tilde u^\eps_\alpha$, we get  
\begin{align*}
|u_\alpha^\eps(x,t)-\tilde u_\alpha^\eps(x,t)|&\leq |u^\eps_\alpha(x,t)-  u^\eps_\alpha(y_\eps,t)| +
 |\tilde u^\eps_\alpha(x,t)- \tilde u^\eps_\alpha(y_\eps,t)|
 \\
 &\leq L|x-y_\eps|+\eps+L|x-y_\eps|\leq (2L+1)\eps.
 \end{align*}
By Ascoli-Arzel\'a Theorem, up to subsequences $\tilde u^\eps_\alpha\to u$  uniformly,  then also $u^\eps_\alpha$ converges uniformly to the same function, which, by \eqref{s1} and \eqref{s2}, satisfies \eqref{estlip} if $\alpha>0$ and \eqref{estlip2} if $\alpha=0$.   

Finally, 
condition \eqref{mono} is equivalent to require that  
$u_0(x+\eta r)-\delta r\geq u_0(x)$ for all $x\in \R^n$ and every $r>0$.
Let us fix $r>0$ and define, for all $\eps>0$,
 the function 
 \[v^\eps(x,t):= u^\eps_\alpha(x+\eta r,t)-\left[\frac{\delta r}{\eps}\right]\eps .\] 
 Then $v^\eps(x,0)=  u_0(x+\eta r)-\left[\frac{\delta r}{\eps}\right]\eps \geq 
 u_0(x+\eta r)-\delta r\geq u_0(x)$ for every $x\in\R^n$. 
 Moreover, by the periodicity of $g$, $v^\eps$ it is also a solution to equation in \eqref{eq}. So, by comparison principle
 we get 
 \[
v^\eps(x,t)=
 u^\eps_\alpha(x+\eta r,t)-\left[\frac{\delta r}{\eps}\right]\eps \geq u^\eps_\alpha(x,t)\qquad \forall (x,t).
 \] 
Passing to the limit as  $\eps\to 0$, we obtain that 
\[u(x+\eta r,t)-\delta r \geq u (x,t)\qquad \forall (x,t),\] 
which gives the thesis.
\end{proof} 

We now recall a well-known result of the theory of viscosity solutions (see \cite{bcd}).
\begin{proposition}\label{limitprop}  Let $\bar c:[0, +\infty)\to [0, +\infty)$ be a continuous function, and let $u_0$ as in \eqref{id}. Then there exists a unique  Lipschitz continuous viscosity solution
to \[\begin{cases} u_t-  \bar c(|\nabla u|) =0&  \text{in }  \R^n\times (0, +\infty)\\
u(x,0)=u_0(x) & \text{in }  \R^n.\end{cases}\] \end{proposition}

\section{Cell problem and asymptotic speed of propagation}\label{seccell} 
We make the formal ansatz that the solution to \eqref{eq} satisfies the following asymptotic expansion:
\begin{equation}\label{ex2} u^\eps_\alpha(x,t)= \eps \chi\left(\frac{u(x,t)}{\eps} \right)\end{equation} where the function $\chi:\R\to \R$ is such that 
 \begin{equation}\label{limitchi} \lim_{z\to \pm \infty} \frac{\chi (z )}{z}=1.\end{equation}  
Note that under this condition, \eqref{ex2} implies that $u^\eps_\alpha\to u$ locally uniformly as $\eps\to 0$.  
 
We compute  \[\begin{cases}(\ue_\alpha)_t= \chi_z u_t +O(\eps)\\
\nabla \ue_\alpha =  \chi_z \nabla u  +O(\eps )\\
 \nabla^2 \ue_\alpha = \frac{1}{\eps}h_{zz} \nabla u \otimes \nabla u +    \chi_{zp} \nabla^2 u \nabla u+ \chi_{z } \nabla^2u +O(\eps)    .
\end{cases}\]

Plugging the asymptotic expansion \eqref{ex2} into the equation \eqref{eq} and   putting $z= \frac{u(x,t)}{\eps}$, $p=\nabla u(x,t)$, $c=u_t(x,t)$ 
 we get the following cell problem, for every $\eps>0$ and $\alpha\geq 0$: 
for every $p\in\R^n$,  show that there exists  a unique constant $\bar c^\eps_\alpha(p)=\bar c_\alpha^\eps(|p|)$ such that the following problem has a solution $\chi=\chi^\eps_p(\cdot)$
\begin{equation}\label{cella} \begin{cases}  
\eps^{\alpha-1} \chi''(z) |p|^2 -\bar c_\alpha^\eps(|p|)\chi'(z) +g(\chi(z))=0 & z \in (0,1)\\ \chi(1)=\chi(0)+1 & \\ \chi'(0)=\chi'(1).\end{cases} \end{equation}
Note that if $\chi_p^\eps$ is a solution to \eqref{cell}, also $\chi_p^\eps(z+k) $ and $\chi_p^\eps(z)+n$ are  solutions for every $k\in \R$ and $n\in \Z$. 
Moreover the function $\chi_p^\eps$ extended to $\R$ satisfies \eqref{limitchi}. Observe that if $g$ is constant, that is $g\equiv \bar g$, then $\bar c_\eps(|p|)=\bar g$ for every $p$ and every $\eps$ and  $\chi_p^\eps(z)= z$. 

Finally, note that the cell problem \eqref{cella} can be reformulated in a more standard way  as follows.  Given $p\in\R^n$, $\alpha\geq 0$, $\eps>0$,  find the constant $\bar c_\alpha^\eps(|p|)$ for which the equation
\begin{equation}\label{cellveroa}  -|p|^2 \eps^{\alpha-1} w''(z)+\bar c_\alpha^\eps(|p|) (w'(z)+1) -g(w(z)+z)=0 \end{equation} admits a periodic solution $w_p^\eps$.
Given a solution  $w_p^\eps$ to \eqref{cellveroa}, defining  $\chi_p^\eps(z)= w_p^\eps(z)+z$, we obtain a solution to \eqref{cella} and viceversa.

\subsection{Case $\alpha=1$, effective Hamiltonian}
In this section we consider the case $\alpha=1$. Under this assumption, the cell problem reads as follows:
for every $p\in\R^n$,  show that there exists  a unique constant $\bar c(|p|)$ such that  there exists a solution $\chi=\chi_p(\cdot)$
\begin{equation}\label{cell} \begin{cases}   \chi''(z) |p|^2 -\bar{c}(|p|)\chi'(z) +g(\chi(z))=0 & z \in (0,1)\\ \chi(1)=\chi(0)+1 & \\ \chi'(0)=\chi'(1).\end{cases} \end{equation}
We can also state the cell problem using the equivalent formulation: given $p\in\R^n$,    find the constant $\bar c(|p|) $ for which there exists a periodic solution $w_p$ to 
\begin{equation}\label{cellvero}  -|p|^2   w''(z)+\bar c(|p|)(w'(z)+1) -g(w(z)+z)=0. \end{equation} 

In the following theorem we show that the cell problem has a (unique) solution.
\begin{theorem}\label{corrector} 
For every $p$ there exists a unique $\bar c(|p|)$ such that there exists a monotone increasing solution $\chi_p$ to  
\eqref{cell}, which is also unique up to horizontal translations.
\\ Moreover, the map $|p|\mapsto \bar c(|p|)$ is continuous, increasing  and positive, \begin{equation}\label{cp}\bar c(|p|) =\begin{cases} 
\int_0^1 g(\chi_p(z))dz=\frac{\int_0^1 g(s)ds }{ \int_0^1 (\chi_p'(z))^2 dz} & p\neq 0\\ \left(\int_0^1 \frac{1}{g(s)} ds\right)^{-1} &  p=0.\end{cases}\end{equation} 
\\ In particular   \begin{eqnarray} \label{limitcorrector0}   & \lim_{|p|\to 0} \bar c(|p|)=\left(\int_0^1 \frac{1}{g(s)} ds\right)^{-1} & \text{ and }  \lim_{|p|\to 0 } \chi_p(z)=\chi_0(z)\text{  in $C(\R)$}\\ \label{limitcorrectorinfty} & \lim_{|p|\to +\infty} \bar c(|p|)=\int_0^1 g(s) ds & \text{ and }   \lim_{|p|\to +\infty} \chi_p(z)=z   \text{  in $C^1(\R)$}, \end{eqnarray} with $c(|p|)<\int_0^1 g(s)ds$ if $g$ is nonconstant.
\end{theorem}

\begin{proof} 
The proof is divided in several steps.  

\noindent {\bf Step 1: construction of a solution for $p=0$. } 

For $p=0$, we rewrite \eqref{cell} as follows
\begin{equation}\label{cell0} \begin{cases} \bar c(0)\chi'(z) -g(\chi(z))=0 & z \in (0,1)\\ \chi(1)=1, \ \chi(0)=0 & \\ \chi'(0)=\chi'(1).\end{cases} \end{equation} 
We integrate the equation between $0$ and $1$ and we get 
\[\int_{0}^{1} \frac{d\chi}{g(\chi)} =\frac{1}{\bar c(0)}\]
which gives the representation formula \eqref{cp}, and the uniqueness of $\bar c(0)$. 
 The solution $\chi_0$ is defined implicitly by  the formula \[\int_{0}^{\chi_0(z)} \frac{ds}{g(s)} =z\int_{0}^{1} \frac{ds}{g(s)} .\]

\noindent {\bf Step 2: construction of a solution for $p\neq 0$.  }

For $|p|\neq 0$, we perform the change of variable $\chi_p(z)=-h\left(-\frac{z}{|p|}\right)$, so the cell problem
\eqref{cell} reads
\begin{equation*} \begin{cases} h''(z)  +c(|p|)h'(z)- g(h(z))=0 & z\in \left(-\frac{1}{|p|},0 \right)\\ h(0)=0,\ \ h\left(-\frac{1}{|p|} \right)=-1 & \\
h'\left( -\frac{1}{|p|} \right)=h'(0),\end{cases} \end{equation*} 
where $c(|p|)=\overline{c}(|p|)/|p|$,
which is equivalent to 
\begin{equation}\label{cello} 
\begin{cases} h''(z)  +c(|p|)h'(z)- g(h(z))=0 & z\in \left(0,\frac{1}{|p|} \right)\\ h(0)=0,\ \ h\left(\frac{1}{|p|} \right)=1 & \\h'\left( \frac{1}{|p|} \right)=h'(0).\end{cases} \end{equation} 

Given   $c>0$ and $a>0$, let $\ha$ be  the unique solution to the ODE:
\begin{equation}\label{odeaiprodi} \begin{cases}  \ha''(z) +c\ha'(z) - g(\ha(z))=0 & z>0\\ 
\ha(0 )=0 &\\
\ha'(0)=a.
 \end{cases} \end{equation} 
Integrating \eqref{odeaiprodi}, for all $z>0$ we get the estimate
\begin{equation}\label{lasagne}
0 < a e^{-cz} + \frac{\min g}{c}\left( 1-e^{-cz}\right)\le \ha'(z)\le a e^{-cz} + \frac{\max g}{c}\left( 1-e^{-cz}\right).
\end{equation}
Let $\bar z :=\sup\{z:\ \ha(z)<1\}\in (0, +\infty)$.
Notice that from \eqref{lasagne} it follows that for  
$c$ small enough there holds $\ha'(z)>a$ for all $z>0$, whereas 
for $c$ big enough we have $\ha'(z)<a$ for all $z>0$.
As a consequence, for all $a>0$ there exists $c(a)>0$ such that 
\begin{equation}\label{millefoglie}
\frac{\min g}{a}\le \,c(a)\le \frac{\max g}{a} \qquad h_{a,c(a)}'\left( \bar z(a)\right)=a.
\end{equation}

From \eqref{lasagne} and \eqref{millefoglie} it also follows that 
\begin{equation}\label{parmigiana}
\frac{\min g}{c(a)}\le h_{a,c(a)}'(z)\le \frac{\max g}{c(a)}\qquad \forall z.
\end{equation}
Since $\int_0^{\bar z(a)}h_{a,c(a)}'(z)dz=1$, \eqref{parmigiana} yields 
\[
\frac{c(a)}{\max g}\le \bar z(a)\le \frac{c(a)}{\min g}\,,
\]
which gives
\begin{equation}\label{pastalforno}
\frac{\min g}{\max g}\,\frac 1a\le \bar z(a)\le \frac{\max g}{\min g}\,\frac 1a\,.
\end{equation}
In particular, there holds
\[
\lim_{a\to 0}\bar z(a)=+\infty\qquad \text{and}\qquad \lim_{a\to +\infty}\bar z(a)=0.
\]
Hence for all $|p|>0$ there exists at least one $a(|p|)$ such that $\bar z(a(|p|))=1/|p|$,
and the solution of \eqref{odeaiprodi} with $a=a(|p|)$ and $c=c(a(|p|))$ is also a solution of \eqref{cello}.

\smallskip
\noindent {\bf Step 3: uniqueness of $\bar c(|p|)$ and $\chi_p$.} 

The case $p=0$ has already been considered in Step 1. Assume by contradiction that there exists $p\in \R^n$, $p\neq 0$,  such that the problem \eqref{cellvero} admits two periodic  solutions $w_1, w_2$, with constants $c_1<c_2$. 
Let $\bar z$ a minimum point of $w_1-w_2$. Note that if $w$ is a periodic solution to \eqref{cellvero}, then $\tilde w (z)= w(z+k)+k$ is still a periodic solution of the same equation for all $k\in\R$.
So we can assume that $w_1(\bar z)=w_2(\bar z)$ and $w_1'(\bar z)=w_2'(\bar z)$. 

At this minimum point, recalling that $\chi'(z)=w'(z)+1> 0$, we have
\[0=-|p|^2 w_1''(\bar z)+c_1 (w_1'(\bar z)+1)-g(w_1(\bar z)+\bar z)\leq -|p|^2 w_2''(\bar z)+c_1 (w_2'(\bar z)+1)-g(w_2(\bar z)+\bar z)\]\[<-|p|^2 w_2''(\bar z)+c_2 (w_2'(\bar z)+1)-g(w_2(\bar z)+\bar z) =0\] which gives a contradiction and proves the uniqueness of $\bar c(|p|)$.

Let now $w_1, w_2$ be two solutions to \eqref{cellvero}, as above, with $w_1(\bar z)=w_2(\bar z)$ and $w_1'(\bar z)=w_2'(\bar z)$ for some $\bar z$.
By uniqueness of solutions to the Cauchy Problem associated to \eqref{cellvero}, it follows that $w_1=w_2$,
which yields the uniqueness of $\chi_p$ up to horizontal translations.

\smallskip 
\noindent {\bf Step 4: properties of $\bar c(|p|)$.}  

Note that integrating the equation \eqref{cell} in $(0,1)$ we get $\bar c(|p|)= \int_0^1 g(\chi_p(z))dz $ and from integrating \eqref{cell}  multiplied by $\chi'_p$ we get 
$\bar c(|p|)\int_0^1 (\chi_p'(z))^2 dz= \int_0^1 g(s)ds$, and then the representation formulas \eqref{cp}. 
In particular from $\bar c(p)=\int_0^1 g(\chi_p(z))dz$ we deduce that $\bar c(p)\geq \min g>0$. Moreover note that, if $g$ is nonconstant, then $\chi_p'(z)$ cannot be constant and   
\[ \int_0^1 (\chi_p'(z))^2 dz > \left(\int _0^1 \chi_p'(z) dz\right)^2 =1.\] So, by \eqref{cp}, we deduce that $\bar c(p)< \int_0^1 g(s)ds$ for every $p$. 

We prove continuity in $0$, since continuity in $p\neq 0$ is much simpler and follows the same argument. Let  $|p_n|\to0$, with $|p_n|\neq 0$ for every $n$. So $\bar c(p_n)$ is a bounded sequence and, by \eqref{cp}, \[\int_0^1 |\chi'_{p_n}(z)|^2 dz  \leq \frac{\int_0^1 g(s)ds }{\min g}.\] We recall that $\chi_{p_n}(z)\in [0,1]$ for $z\in [0,1]$, so this estimates give an apriori bound in $H^1(0,1)$ for $\chi_{p_n}$.  So, up to passing to a subsequence, we get that 
$\bar c(p_n)\to \tilde c$ and $\chi_{p_n}\to \chi$  locally uniformly. Then, by stability of viscosity solutions, $\chi$ is a solution to \eqref{cell0}, and by uniqueness $\tilde c=\bar c(0)$ and $\chi=\chi_0$. Moreover, since both $\chi_{p_n}(z)-z$ and $\chi_0(z)-z$ are periodic functions such that their difference converges locally uniformly to $0$, then we can conclude using periodicity that the convergence is uniform on $\R$. This gives  \eqref{limitcorrector0}. 

Now we prove \eqref{limitcorrectorinfty}. Reasoning as above, we get that $\bar c(|p|)$ and  $\chi'_p$ are equibounded respectively in $\R$ and in $L^2(0,1)$, uniformly with respect to $|p|$. By equation \eqref{cell} we get  $\chi''_p=\frac{\bar c(|p|)\chi'-g(\chi)}{|p|^2}$, so  the uniform $L^2$ bound on $\chi'_p$ implies  an uniform $L^2$ bound on 
$\chi''_p$, uniform in  $|p|>1$. Eventually passing to a subsequence,  $\bar c(|p|)\to c$,  $\chi_p\to \chi $ and $\chi'_p\to \chi'$ locally uniformly  as $|p|\to +\infty$. 
By stability of viscosity solutions,  we get that $\chi$ solves $ \chi''(z)=0$, with $\chi(0)=0$, $\chi(1)=1$. So $\chi(z)=z$. Moreover since $\chi_p(z)-z$ is a periodic function converging locally uniformly  in $C^1$ to $0$, we get that actually it converges uniformly in $C^1$ to $0$ in the whole $\R$. 
Therefore $g(\chi_p(z))\to g(z)$ uniformly and then  we conclude, using \eqref{cp}, that $\lim_{|p|\to +\infty} \bar c (|p|)=\lim_{|p|\to +\infty} \int_0^1 g(\chi_p(z))dz = \int_0^1 g(s) ds$.

Finally we prove monotonicity of $\bar c(|p|)$. Assume by contradiction that  there exist $p_1, p_2\in \R^n$, $|p_1|>|p_2|$,  such that  $\bar c(|p_1|)<\bar c(|p_2|)$. Let $w_1, w_2$ two solutions to \eqref{cellvero} associated to $p_1,p_2$. 
Let $\bar z$ a minimum point of $w_1-w_2$, reasoning as in Step 3, we can assume $w_1(\bar z)=w_2(\bar z)$. Then, at this minimum point, 
\begin{eqnarray*} 0 &=& -|p_1|^2 w_1''(\bar z)+\bar c(|p_1|)( w_1'(\bar z)+1)-g(w_1(\bar z)+\bar z)\\ &<& -|p_2|^2 w_2''(\bar z)+\bar c(|p_2|) (w_2'(\bar z)+1)-g(w_2(\bar z)+\bar z)=0\end{eqnarray*} which gives a contradiction. Therefore $c(|p_1|)\geq c(|p_2|)$. 
\end{proof}

%

\begin{remark} \upshape We expect that the same result  holds also for $\int_0^1  g(s)ds>0$. In this case though the ODE arguments are much more involved. 
Observe that, if $g$ changes sign, then necessarily we have $\bar c(0)=0$ and  $\chi_0(z)\equiv s_0$ for $z\in (0,1)$, where $s_0\in [0,1]$ is such that $g(s_0)=0$. 

In the limiting case that $\int_0^1 g(s)ds=0$, the same cell problem has been solved in \cite{j}, see also  \cite[Prop. 1.3]{m}, showing that  
 there exists a solution to \eqref{cell} with  $\bar c(|p|)\equiv 0$ for every $p$.\end{remark}

%
%
%
%
%
\subsection{Case $\alpha\neq 1$,  the weak and strong diffusion regimes} 
In this section we analyze the solution of the cell problem \eqref{cell} in the case $\alpha\neq 1$. 

The solution to the cell problem is an easy corollary to Theorem \ref{corrector}. 
Moreover, we can also compute the asymptotic behavior as $\eps\to 0$ to the 
solutions to the cell problem. 
\begin{proposition} \label{cellalpha} 
Let $\alpha\neq 1$ and $\eps>0$.  Then  there 
exists a unique constant $\bar c_\alpha^\eps(|p|)$ such that \eqref{cella} admits a solution $\chi_p^\eps$, 
which is monotone increasing, and unique up to horizontal translations. 

Moreover, 
\begin{itemize}\item[i)] if $\alpha>1$ 
then we have that \[\lim_{\eps\to 0^+} \bar c_\alpha^\eps(|p|)=\bar c_{+}(|p|):= \left(\int_0^1 g^{-1}(s)ds\right)^{-1} \qquad \forall p\in \R^n\]  and $\chi_p^\eps\to \chi_0$  uniformly in $C(\R)$, for every $p$, where $\chi_0$ is the solution to \eqref{cell0};

\item[ii)] if $\alpha<1$ 
then we have that \[\lim_{\eps\to 0^+}\bar c_\alpha^\eps(|p|)=\bar c_{-}(|p|):= \begin{cases} \left(\int_0^1 g^{-1}(s)ds\right)^{-1}  & p=0 \\ 
\int_0^1 g(s)ds & p\neq 0 \end{cases}\]  and, for $p\neq 0$,  $\chi_p^\eps(z)\to z$ uniformly in $C^1(\R)$, whereas $\chi_0^\eps=\chi_0$. 
\end{itemize} 
\end{proposition}
\begin{proof} 
Note that \eqref{cella}  coincides with the cell problem \eqref{cell}
associated to $p_\eps= p\eps^{\frac{\alpha -1}{2}}$. Therefore by  uniqueness of $\bar c$ proved in Theorem \ref{corrector}, for every $\eps>0$ and every $\alpha\neq 1$, there 
exists a unique $\bar c_\alpha^\eps(|p|)=\bar c(|p|\eps^{\frac{\alpha -1}{2}})$, such that there exists a solution $\chi_p^\eps$ to \eqref{cella}. Note that $\chi^\eps_p= \chi_{p\eps^{\frac{\alpha -1}{2}}}$. 

Moreover,  if $\alpha>1$, since $|p|\eps^{\frac{\alpha-1}{2}}\to 0$ for every $p$, then by \eqref{limitcorrector0}, $\bar c_\alpha^\eps(|p|)\to \bar c(0)$ and  $\chi_p^\eps\to \chi_0$ uniformly.

If $\alpha<1$, then for $p\neq 0$, $|p|\eps^{\frac{1-\alpha}{2}} \to +\infty$ and then, by \eqref{limitcorrectorinfty},  $\bar c_\alpha^\eps(|p|)\to \int_0^1 g(s)ds$ and $\chi_p^\eps(z)\to z$  uniformly in $C^1$ for $p\neq 0$ as $\eps\to 0$. 

\end{proof} 

\section{Convergence of solutions} \label{secconv}
In this section we study the asymptotic limit as $\eps\to 0$ of the solutions to \eqref{eq} in the different regimes, $\alpha=1$, $\alpha>1$, $0<\alpha<1$ and $\alpha=0$.
 
According to Proposition \ref{seq},  the solutions $u^\eps_\alpha$ to \eqref{eq}  converge locally  uniformly, up to subsequences,  to a Lipschitz function $u$. 
Our aim is to show that the  limit $u$ is a viscosity solution of an effective equation, given by $u_t- c(|\nabla u|)=0$. The effective operator has been defined in Theorem \ref{corrector} for $\alpha=1$, and it coincides with  the continuous function  $\bar c(|p|)$. In the case $\alpha\neq 1$, the effective operator  has been  defined in Proposition \ref{cellalpha}. It coincides in the case  $\alpha>1$ with the constant value $\bar c_+(|p|)\equiv  \left(\int_0^1 (g(s))^{-1}ds\right)^{-1}$, whereas in the case $0<\alpha<1$, it is  $\bar c_-(|p|)$, which coincides with $\int_0^1 g$ for $p\neq 0$, and with  $ \left(\int_0^1  (g(s))^{-1}ds\right)^{-1}$ for $p=0$. We consider  also the limiting case $\alpha=0$, where the effective equation is given by $u_t-\bar F(\nabla u, \nabla^2 u)=0$. 

We start with a preliminary estimate which follows from the comparison principle for \eqref{eq} .

\begin{proposition}\label{corollario2}
Let $\ue_\alpha$ be  the  solution   to \eqref{eq} with $\alpha\geq 0$.  Then every uniform limit $u$ of $u^\eps_\alpha$ satisfies 
\[ \inf_{\R^n} u_0 + t \left(\int_0^1 \frac{1}{g(s)}ds\right)^{-1}\le 
u(x,t)\le \sup_{\R^n} u_0 + t \left(\int_0^1 \frac{1}{g(s)}ds\right)^{-1}.\] 
\end{proposition} 
\begin{proof}  
It is enough to prove the result when $u_0\equiv k$, for some constant $k\in\R$. The thesis then follows by comparison principle for \eqref{eq}.

Recall that  $\bar c(0)=\left(\int_0^1 \frac{1}{g(s)}ds\right)^{-1}$ and observe that, if $\chi_0$ is the solutions to \eqref{cell0},  the functions 
\[v^\eps(x, t)=\eps\chi_0\left(\frac{t\bar c(0)}{\eps}\right)+\eps\left[\frac{k}{\eps}\right]-\eps\quad\text{and}\quad 
V^\eps(x, t)=\eps\chi_0\left(\frac{t\bar c(0)}{\eps}\right)+\eps\left[\frac{k}{\eps}\right]\]  
are respectively a sub and a supersolution to \eqref{eq}, for every $\alpha\geq 0$. So, by comparison 
\[v^\eps(x,t)\leq u^\eps_\alpha(x,t)\leq V^\eps(x,t).\] Letting $\eps\to 0$ and recalling that $\chi_0(z)/z\to 1$ as $z\to +\infty$, we get the conclusion. 
\end{proof}

\subsection{Case $\alpha=1$}
\begin{theorem} \label{conv1} Let $\ue$ be  the  solution   to \eqref{eq} for $\eps>0$ and $\alpha=1$. Then $\ue$ converges as $\eps\to 0$ locally uniformly to the unique Lipschitz continuous viscosity solution to \begin{equation}\label{limit} \begin{cases} 
u_t-\bar c(|\nabla u|)=0, \\ u(x,0)=u_0(x).  \end{cases} \end{equation} 
\end{theorem} 
\begin{proof} 
By Proposition \ref{seq}, up to passing to subsequences $\ue\to u$ locally uniformly, where $u$ is a Lipschitz continuous function which satisfies \eqref{estlip}. 
So, if we prove that $u$ is a solution to \eqref{limit}, we conclude using uniqueness of solutions to \eqref{limit} as stated in Proposition \ref{limitprop} the convergence of the whole sequence $u^\eps$ to $u$. 

We show that $u$ is a subsolution to the effective equation in \eqref{limit}, the proof of the supersolution property being completely analogous.

Let $(x_0,t_0)$ and $\phi$ a smooth function such that $u-\phi$ has a strict maximum at $(x_0,t_0)$ and  $u(x_0,t_0)=\phi(x_0,t_0)$.  Let $R>0$ and let $\bar B$ the closed ball centered at $(x_0,t_0)$ and with radius $R$.   Define a family of perturbed test functions, parametrized by a parameter $s\in\R$,  as follows:
\[\phi^\eps_s(x,t)=\eps \chi_p\left(\frac{\phi(x,t)}{\eps}+s\right) \] where $\chi_p$ is a solution to \eqref{cell} with $p=\nabla\phi(x_0,t_0)$. By the properties of $\chi_p$, $\phi^\eps_{s+1}(x,t)=\phi^\eps_s(x,t)+\eps$.
Note that $\phi^\eps_s \to \phi$ as $\eps\to 0$, locally uniformly in $x,t,s$. So for every $s$ there exists a sequence $(x^\eps_s, t^\eps_s)\to (x_0,t_0)$ as $\eps\to 0$ 
such that $(x^\eps_s, t_s^\eps)$ is a maximum point for $\ue-\phi_s^\eps$ in $\bar B$ and $(\ue-\phi_s^\eps)(x_s^\eps, t_s^\eps)\to  u(x_0,t_0)-\phi(x_0,t_0)=0$.
We claim that for every $\eps>0$ we can choose $s^\eps$ such that $(\ue-\phi_{s^\eps}^\eps)(x_{s^\eps}^\eps, t_{s^\eps}^\eps)=0$. 
Indeed, let $m(s)=\max_{\bar B} (\ue-\phi_{s}^\eps)$. 
Note that  $m(s)$ is continuous and $m(s+k)=m(s)-\eps k$ for every $k\in \Z$. Therefore by continuity there exists $s^\eps$ such that $m(s^\eps)=0$. 

From now on we fix the test function $\phi^\eps=\phi^\eps_{s^\eps}$ and the maximum point $(x_{s^\eps}^\eps, t_{s^\eps}^\eps)=(x^\eps, t^\eps)$. So, $\ue(x^\eps, t^\eps)=\phi^\eps(x^\eps, t^\eps)$, $\ue\leq \phi^\eps$ in $\bar B$ and $(x^\eps, t^\eps)\to (x_0,t_0)$ as $\eps\to 0$. Indeed, let $\tilde s_\eps\in [0,1)$ be the fractional part of $s^\eps$, then by the properties of $\chi_p$ we get that $(x^\eps_{\tilde s^\eps}, t^\eps_{\tilde s^\eps})= (x^\eps_{s^\eps}, t^\eps_{s^\eps})$.  So the conclusion follows by the locally uniform convergence of $\phi^\eps_{s^\eps}$ to $\phi$. 

Let us denote $z^\eps=\frac{\phi(x^\eps,t^\eps)}{\eps}+s^\eps $, so that $u^\eps(x^\eps, t^\eps)=\eps \chi_p(z^\eps)$. 
We compute  \[\ue_t(x^\eps, t^\eps)= \phi_t^\eps(x^\eps, t^\eps)=\chi'_p(z_\eps) \phi_t(x^\eps, t^\eps)\] and  
\[-\eps\Delta \ue(x^\eps, t^\eps)\geq -\eps\Delta \phi^\eps(x^\eps, t^\eps)= - \eps\chi'_p(z_\eps) \Delta \phi(x^\eps, t^\eps)-\chi''_p(z^\eps)  |\nabla\phi(x^\eps, t^\eps)|^2.\] 
Plugging these quantities   into  Equation \eqref{eq} computed at $(x^\eps, t^\eps)$, we obtain 
\[ 0=\ue_t-\eps \Delta \ue -g\left(\frac{\ue}{\eps}\right) \geq \chi'_p(z^\eps) \phi_t(x^\eps, t^\eps)- \eps\chi'_p(z^\eps) \Delta \phi(x^\eps, t^\eps)-\chi''_p(z^\eps)  |\nabla\phi(x^\eps, t^\eps)|^2- 
g(\chi_p(z^\eps)).\]
Using the fact that $\chi_p$ solves \eqref{cell}, we get 
\begin{eqnarray}
 \nonumber
 0 &\geq & \chi'_p(z_\eps) \left(\phi_t(x_0, t_0)- \bar c(|\nabla\phi(x_0,t_0)|) \right) \\
\label{due} &-& \chi'_p(z_\eps) \left(\phi_t(x_0, t_0)-\phi_t(x^\eps, t^\eps)   +\eps  \Delta \phi(x^\eps, t^\eps)\right) \\ \label{tre} 
&-& \chi''_p(z^\eps) \left( |\nabla\phi(x^\eps, t^\eps)|^2-|\nabla\phi(x_0,t_0)|^2\right).\end{eqnarray} 

Computing \eqref{cell} at minima and maxima of $\chi'_p$ we deduce  that \begin{equation}\label{stima1} \chi_p'(z)\in \left[\frac{\min g}{\int_0^1 g(s)ds}, \frac{\max g}{\min g}\right]\qquad\forall p, \ \ \forall z.\end{equation} Moreover, from equation \eqref{cell}, we deduce that also \begin{equation} \label{stima2}\|\chi''_p\|_\infty\leq \frac{\max g-\min g}{|p|^2 }\quad \text{ if }p\neq 0\qquad\text{and}\quad  \|\chi''_0\|_\infty \leq \frac{ \|g\|_\infty \|g'\|_\infty}{\bar c(0)^2}. \end{equation} Therefore, as $\eps\to 0$, we get that the terms in \eqref{due}, \eqref{tre} go to zero by the smoothness of $\phi$, and we are left with $ \phi_t(x_0, t_0)- \bar c(|\nabla\phi(x_0,t_0)|)\leq 0$. 
\end{proof}  

\subsection{Case $\alpha>1$} 
\begin{theorem} \label{conv2} Let $\ue_\alpha$ be  the  solution   to \eqref{eq} with $\alpha>1$. 
Then \[\lim_{\eps\to 0} \ue_\alpha(x,t)= u_0(x)+t \left(\int_0^1 \frac{1}{g(s)}ds\right)^{-1} \qquad \text{  locally uniformly.}\]
\end{theorem} 
\begin{proof} 
The argument is similar (in fact easier) of that in the proof of Theorem \ref{conv1}. We sketch it briefly. Up to subsequences, we know that
$u^\eps_\alpha$ is  converging locally uniformly to some function $u$ (eventually depending on the subsequence). 

We show that $u_t\leq \left(\int_0^1 \frac{1}{g(s)}ds\right)^{-1}$  in the viscosity sense. A completely analogous argument shows that $u_t\geq \left(\int_0^1 \frac{1}{g(s)}ds\right)^{-1}$ in the viscosity sense. Recalling that $u(x,0)=u_0(x)$, we conclude that therefore   $u(x,t)=u_0(x)+t \left(\int_0^1 \frac{1}{g(s)}ds\right)^{-1}$.

Let $(x_0,t_0)$ and $\phi$ a smooth function such that $u-\phi$ has a strict maximum at $(x_0,t_0)$ and  $u(x_0,t_0)=\phi(x_0,t_0)$.  Let $R>0$ and let $\bar B$ the closed ball centered at $(x_0,t_0)$ and with radius $R$.   We define  a perturbed test function as follows:
\[\phi^\eps(x,t)=\eps \chi_0\left(\frac{\phi(x,t)}{\eps}+s\right) \] where $\chi_0$ is the  solution to \eqref{cell0} and the parameter $s$ is chosen as in  the proof of Theorem \ref{conv1}. So, $(x^\eps, t^\eps)$ is a maximum point for $\ue-\phi^\eps$ in $\bar B$ and $\ue(x^\eps, t^\eps)=\phi^\eps(x^\eps, t^\eps)$, $\ue\leq \phi^\eps$ in $\bar B$ and $(x^\eps, t^\eps)\to (x_0,t_0)$ as $\eps\to 0$. 

Let us denote $z^\eps=\frac{\phi(x^\eps,t^\eps)}{\eps}+s  $, so that $u^\eps(x^\eps, t^\eps)=\eps \chi_0(z^\eps)$. 

So using the fact that $(x^\eps, t^\eps)$ is a maximum point for $\ue-\phi^\eps$, we plug $\phi^\eps$ into \eqref{eq} and we obtain 
\[ 0=\ue_t-\eps^\alpha \Delta \ue -g\left(\frac{\ue}{\eps}\right) \geq \chi'_0  \phi_t(x^\eps, t^\eps)- \eps^\alpha\chi'_0\Delta \phi(x^\eps, t^\eps)-\eps^{\alpha-1} \chi''_0  |\nabla\phi(x^\eps, t^\eps)|^2- g(\chi_0).\]
By regularity of $\phi$ and using the estimates \eqref{stima1}, \eqref{stima2}, we get that, as $\eps\to 0$, $\eps^\alpha\chi'_0\Delta \phi(x^\eps, t^\eps) \to 0$ and $\eps^{\alpha-1} 
\chi''_0  |\nabla\phi(x^\eps, t^\eps)|^2\to 0$. So, we conclude recalling that $\chi_0'>0$ and that $\chi_0$ solves \eqref{cell0}
that \[ 0  \geq  \phi_t(x_0, t_0)- \bar c(0)+O(\eps). \]
\end{proof} 

\subsection{Case $0<\alpha<1$}
In this case, the limit differential operator $\bar c_-(|p|)$ is not continuous, but just lower semicontinuous. 
In particular the lower semicontinuous envelope of $\bar c_-$ coincides with the function itself, whereas the upper semicontinuous envelope is the constant 
function $\bar c_-(|p|)^\ast\equiv \int_0^1 g(s)ds$. 

We now show that every limit of $u^\eps_\alpha$ is a viscosity solution of the limit problem 
\eqref{limit01}. According to the definition recalled in Section \ref{notation}, this means the following.
If $\phi$ is  a smooth test function such that  $u(x_0,t_0)=\phi(x_0, t_0)$ and $u\leq \phi$, then  $\phi_t(x_0, t_0)\leq \int_0^1 g(s) ds$. 
If, on the other hand, $u\geq \phi$, then  $\phi_t(x_0, t_0)\geq \bar c_-(|\nabla \phi(x_0, t_0)|)$, 
so in particular $\phi_t(x_0, t_0)\geq  \int_0^1 g(s) ds$ at points where $\nabla \phi(x_0,t_0)\neq 0$ and  
$\phi_t(x_0, t_0)\geq \bar c_-(0)= ( \int_0^1 g^{-1}(s) ds)^{-1}$  at points where $\nabla \phi(x_0,t_0)= 0$.   

We recall that due to the discontinuity of the operator, differently to the case $\alpha\ge 1$, viscosity solutions to \eqref{limit01} are in general  not unique.

\begin{theorem} \label{conv3} Let $\ue_\alpha$ be  the  solution   to \eqref{eq} with $0<\alpha<1$.  
Every locally uniformly limit $u$ of $u^\eps_\alpha$ is a Lipschitz continuous function, which satisfies \eqref{estlip}, and solves in  the viscosity sense the problem
 \[\begin{cases} u_t-\bar c_-(\nabla u)=0  & \text{ in }\R^n\times (0+\infty)\\ 
u(x,0)=u_0(x) & \text{ in  }\R^n.\end{cases}\]
Moreover, 
\begin{itemize} 
\item[i)]  $u$ satisfies in the viscosity sense   \[u_t= \int_0^1g(s)ds\]   in every open set $\Omega\subset \R^n\times [0,+\infty)$, such that $\nabla u \neq 0$ a.e. in  $\Omega$.  
\item[ii)]  $u$ is a viscosity subsolution to \[u_t= \left(\int_0^1 \frac{1}{g(s)}ds\right)^{-1}\] at every point $(x_0,t_0)$ such that  $(0, X, \lambda)\in D^+ u(x_0, t_0)$ with $X< 0$ in the sense of matrices.
\item [iii)]  $u$ is a viscosity supersolution to \[u_t= \int_0^1 g(s)ds\] at every point $(x_0,t_0)$ such that  $(0, X, \lambda)\in D^- u(x_0, t_0)$,  and there exist $\eta\in \R^n\setminus \{0\}$, $\delta>0$,  such that $\eta^t X\eta\ge \delta|\eta|^2$.
\end{itemize} 
\end{theorem} 

\begin{proof} 
The fact that, up to a subsequence, $\ue_\alpha$ converges locally uniformly to a Lipschitz function $u$ is proved in Proposition \ref{seq}.
Since $u$ is Lipschitz continuous, then it is differentiable almost everywhere. 

We prove now that $u$ is a viscosity solution to the limit problem. 
We show the statement for supersolutions, since for subsolutions is completely analogous. 

Fix $\phi$ a smooth test function such that  $u(x_0,t_0)=\phi(x_0, t_0)$ and $u< \phi$ elsewhere. 
We consider two cases, depending on the value of $\nabla \phi(x_0,t_0)$.


\smallskip

{\bf Case $1$: $\nabla \phi(x_0,t_0)=p\neq 0$.} In this case we shall prove that $\phi_t(x_0, t_0)\geq \int_0^1 g(s)ds$.

Define $p^\eps=\eps^{(\alpha-1)/2}p$ and $\chi_{p^\eps}$ the solution to \eqref{cell}, with $\bar c(\eps^{(\alpha-1)/2}|p|)$. We define the perturbed test function as in the proof of Theorem \ref{conv1}: 
\[\phi^\eps(x,t)=\eps \chi_{p^\eps}\left(\frac{\phi(x,t)}{\eps}+s\right). \]  
Since $\chi_{p^\eps}(z)$ converges uniformly to $z$ as $\eps\to 0$ by Proposition \ref{cellalpha}, we get that $\phi^\eps\to \phi$ locally uniformly for every $s$. Reasoning as in the proof of Theorem \ref{conv1}, we get that there exist $s^\eps$, $x^\eps$ $t^\eps$ such that 
$\ue(x^\eps, t^\eps)=\phi^\eps(x^\eps, t^\eps)$, $\ue\geq \phi^\eps$  and $(x^\eps, t^\eps)\to (x_0,t_0)$ as $\eps\to 0$. 
So, plugging $\phi^\eps$ into Equation \eqref{eq} computed at $(x^\eps, t^\eps)$ we obtain 
\[ 0=\ue_t-\eps^\alpha \Delta \ue -g\left(\frac{\ue}{\eps}\right) \leq \chi'_{p^\eps}  \phi_t- \eps^{\alpha}\chi'_{p^\eps} \Delta \phi-\chi''_{p^\eps}  \eps^{\alpha-1}|\nabla\phi|^2-
g(\chi_{p^\eps}).\]
Using the fact that $\chi_{p^\eps}$ solves \eqref{cell}, we get 
\begin{eqnarray*}  0 &\leq & \chi'_{p^\eps}\left(\phi_t(x_0, t_0)- \int_0^1 g(s)ds\right)- \chi'_{p^\eps}
\left( \bar c(\eps^{\frac{\alpha-1}{2}}|p|)- \int_0^1 g(s)ds\right) \\
 &-& \chi'_{p^\eps} \left(\phi_t(x_0, t_0)-\phi_t(x^\eps, t^\eps)   +\eps^\alpha  \Delta \phi(x^\eps, t^\eps)\right) 
\\
&-&\chi''_{p^\eps} \eps^{\alpha-1} \left( |\nabla\phi(x^\eps, t^\eps)|^2-|\nabla\phi(x_0,t_0)|^2\right).
\end{eqnarray*} 
Using \eqref{stima1}, \eqref{stima2}, \eqref{limitcorrectorinfty} and the regularity of $\phi$, letting $\eps\to 0$
we conclude that $\phi_t(x_0, t_0)\geq \int_0^1 g(s)ds$. 

\smallskip

{\bf Case $2$: $\nabla \phi(x_0,t_0)=0$.} In this case we shall prove that $\phi_t(x_0, t_0)\geq (\int_0^1 g^{-1}(s)ds)^{-1}$.

As in {Case 1}, we let 
\[\phi^\eps(x,t)=\eps \chi_{p}\left(\frac{\phi(x,t)}{\eps}+s^\eps\right), \]  
where $p\in \R^n$ will be determined later. As above, there exist $s^\eps$, $x^\eps$ $t^\eps$, depending continuously  on $p$, such that 
$\ue(x^\eps, t^\eps)=\phi^\eps(x^\eps, t^\eps)$, $\ue\geq \phi^\eps$  and $(x^\eps, t^\eps)\to (x_0,t_0)$ as $\eps\to 0$. 
Plugging $\phi^\eps$ into \eqref{eq}, evaluating the equation at $(x^{\eps}, t^{\eps})$,
and recalling \eqref{cell} we obtain 
\begin{eqnarray*} 
0 &\leq& \chi'_{p}  \phi_t(x^\eps, t^\eps)
- \eps^{\alpha}\chi'_{p} \Delta \phi(x^\eps, t^\eps)-\chi''_{p}  \eps^{\alpha-1}|\nabla\phi(x^\eps, t^\eps)|^2-
g(\chi_{p})
\\ 
&=& \chi'_{p} \left(\phi_t(x^\eps, t^\eps) - \bar c(|p|)  \right) -\eps^\alpha \chi'_{p} \Delta \phi(x^\eps, t^\eps)
- \chi''_{p} \left( \eps^{\alpha-1} |\nabla\phi(x^\eps, t^\eps)|^2-|p|^2\right).
\end{eqnarray*}
We now consider two subcases:

\smallskip

{\bf Case $2\,a$: for $p=0$ we have $\eps^{\alpha-1} |\nabla\phi(x^\eps, t^\eps)|^2\to 0$, up to a subsequence as $\eps\to 0$.} We choose $p=0$ and we get
\[
0\le \chi'_{0} \left(\phi_t(x^\eps, t^\eps) - \bar c(0)  \right) 
-\eps^\alpha \chi'_{0} \Delta \phi(x^\eps, t^\eps)
- \chi''_{0} \eps^{\alpha-1} |\nabla\phi(x^\eps, t^\eps)|^2.
\]
Using the assumption and passing to the limit as $\eps\to 0$, we then get $\phi_t(x_0,t_0)\ge \bar c(0)$.

\smallskip

{\bf Case $2\,b$: for $p=0$ we have $\eps^{\alpha-1} |\nabla\phi(x^\eps, t^\eps)|^2\ge \delta$, for some $\delta>0$ and $\eps$ small enough.}
For $p\ne 0$ we have
\[
0\le \chi'_{p} \left(\phi_t(x^\eps, t^\eps) - \bar c(|p|)  \right) -\eps^\alpha \chi'_{p} \Delta \phi(x^\eps, t^\eps)
- \chi''_{p} |p|^2\left( \frac{\eps^{\alpha-1} |\nabla\phi(x^\eps, t^\eps)|^2}{|p|^2}-1\right).
\]
Notice that, recalling our assumption, we have 
\[
\lim_{|p|\to 0}\frac{\eps^{\alpha-1} |\nabla\phi(x^\eps, t^\eps)|^2}{|p|^2} = +\infty
\qquad \text{and} 
\lim_{|p|\to \infty}\frac{\eps^{\alpha-1} |\nabla\phi(x^\eps, t^\eps)|^2}{|p|^2} = 0.
\]
Then, by a continuity argument, there exists $p^\eps\ne 0$ such that $\eps^{\alpha-1} |\nabla\phi(x^\eps, t^\eps)|^2=|p^\eps|^2$.
For $p=p^\eps$ it then follows 
\[
0\le \chi'_{p} \left(\phi_t(x^\eps, t^\eps) - \bar c(|p^\eps|)  \right) -\eps^\alpha \chi'_{p} \Delta \phi(x^\eps, t^\eps),
\]
which gives $\phi_t(x_0,t_0)\ge \bar c(0)$, in the limit $\eps\to 0$, recalling that $\bar c(|p|)\ge \bar c(0)$ for any $p\in\R^n$.

\smallskip

We now prove assertions $i)$, $ii)$, $iii)$.

{\em Proof of $i)$.} First of all observe that repeating the proof of {Case 1}, we get that  that $u_t=\int_0^1 g(s)ds$ almost everywhere in $\Omega$. If this is true, then $u_t=\int_0^1 g(s)ds$ in the viscosity sense in $\Omega$. Indeed, let $\rho_\delta$ be a sequence of standard mollifiers. So $u_\delta=u\ast \rho_\delta\to u$ uniformly and $(u_\delta)_t=\int (u_t\ast \rho_\delta) =\int_0^1 g(s)ds$ everywhere in $\Omega$. The conclusion then follows from the stability of viscosity solutions. 

\smallskip

{\em Proof of $ii)$.} 
Let  $\phi$ such that $u-\phi$ has a strict maximum at $(x_0,t_0)$, with $\nabla \phi(x_0,t_0)=0$, and   $ \nabla^2\phi(x_0,t_0) <0$ in the sense of matrices.
Then  we show that  $\phi_t(x_0,t_0)\leq  (\int_0^1 \frac{1}{g})^{-1}$. 

We define the function \[\tilde\phi( t)=  u(x_0,t_0)+\phi_t(x_0,t_0)(t-t_0) +C(t-t_0)^2.\]
Choosing appropriately $C$ and using the fact that $\nabla^2\phi(x_0,t_0)<0$, there exists $r,\tau>0$ such that $u(x,t)\le \phi(x,t)\leq \tilde \phi(t)$ for every $(x,t)\in B(x_0,r)\times (t_0-\tau,t_0+\tau)$. 
Observe also that $u(x_0, t_0)=\tilde\phi(t_0)$ and $\tilde\phi_t(t_0)=\phi_t(x_0, t_0)$. 

As above, we let  \[\phi^\eps ( t)= \eps \chi_0\left(\frac{\tilde \phi(t)}{  \eps}+s\right),\]
where $\chi_0$ is the solution to \eqref{cell0}. 
Note that by the properties of $\chi_0$,  we have $\phi^\eps(t)\to \tilde \phi(t)$ 
 locally uniformly as $\eps\to 0$. 
As in the proof of Theorem \ref{conv1},
we choose $s$ such that
there exists $(x^\eps, t^\eps)\to (x_0,t_0)$, with $u^\eps(x^\eps, t^\eps)=\phi^\eps(t^\eps)$ and $u^\eps\leq \phi^\eps$. 

Plugging $\phi_\eps$ into Equation \eqref{eq} computed at $(x^\eps, t^\eps)$, we get 
\[0=u^\eps_t-\eps^\alpha\Delta u^\eps-g\left(\frac{u^\eps}{\eps}\right) 
\geq (\phi^\eps)_t-g\left(\frac{\phi^\eps}{\eps}\right)=\tilde \phi_t(t_0)\chi_0'-g(\chi_0)= \chi_0'(\tilde\phi_t(t_0)-\bar c(0))\] from which we conclude. 

\smallskip

{\em Proof of $iii)$.} Let  $\phi$ such that $u-\phi$ has a strict minimum at $(x_0,t_0)$, 
with $\nabla \phi(x_0,t_0)=0$, and $\eta^t \nabla^2\phi(x_0,t_0)\eta\ge \delta |\eta|^2$.
We shall show that $\phi_t(x_0, t_0)\geq \int_0^1 g(s)ds$.

Let $V$ be a neighborhood of $(x_0,t_0)$ such that $\eta^t \nabla^2\phi(x,t)\eta> \delta>0$ for every $(x,t)\in V$. 
Let $\phi^h(x,t)=\phi(x+h\eta, t)$. Observe that $\phi^h\to \phi$ uniformly as $h\to 0$.  Let $(x_h,t_h)$ a minimum point of $u-\phi^h$ in $V$. Then, eventually passing to a subsequence,  we have 
$(x_h,t_h)\to (x_0,t_0)$. 

Observe that  at points $(x,t)$ where  $\nabla \phi^h(x ,t)=\nabla \phi(x+h\eta , t)=0$, then $u(x ,t )-\phi^h(x ,t )>0$. Indeed 
\[u(x ,t )-\phi^h(x ,t )\geq  \phi(x ,t)-\phi(x+h\eta, t)=  \frac{1}{2} h^2 \eta^t \nabla^2\phi(x+h\eta,t )\eta+o(h^2)>0.\] Since  for  $h$ sufficiently small
\begin{align*}
(u-\phi^h)(x_h,t_h) &\leq u(x_0,t_0)-\phi^h(x_0,t_0)= \phi(x_0,t_0)-\phi(x_0+h\eta, t_0)
\\ &= -\frac{1}{2} h^2 \eta^t \nabla^2\phi(x_0,t_0)\eta+o(h^2)<0,
\end{align*} 
it follows that $\nabla \phi^h(x_h,t_h)\neq 0$ and $u-\phi^h$ has a minimum at $(x_h,t_h)$. Repeating the proof of {Case 1},  we get that $\phi^h_t(x_h,t_h)=\phi_t(x_h+h\eta,t_h)\geq \int_0^1 g$. Letting $h\to 0$ we obtain the result.   
\end{proof} 

From Theorem \ref{conv3} and Corollary \ref{corollario2} we deduce immediately the following estimates. 
\begin{corollary}\label{corollario0}  
Every uniform limit $u$ of $u^\eps_\alpha$ satisfies 
\begin{multline*} u_0(x)+ t\left(\int_0^1 \frac1{g(s)}\,ds\right)^{-1}\leq u(x,t)\\ \leq \min\left( u_0(x)+ t\int_0^1 g(s)ds, \ \sup_{\R^n} u_0 + t \left(\int_0^1 \frac{1}{g(s)}ds\right)^{-1}\right).\end{multline*}
\end{corollary} 
We  now analyze more in detail the behavior of the  limit function for some classes of initial data.
\begin{corollary}\label{corollario1}  
Assume that either $u_0$ is convex and nonconstant
or $u_0$ is unbounded from above and there exists $\eta\in\R^n$, with $|\eta|=1$, 
such that 
\begin{equation}\label{equd}
\nabla u_0(x)\cdot \eta\ge \delta\qquad \text{for some $\delta>0$, and for a.e. $x\in \R^n$}.
\end{equation}
Then the solutions $u^\eps_\alpha$ converge (locally) uniformly to the function
\begin{equation}\label{elpo} u(x,t)=u_0(x)+ t\int_0^1 g(s)ds.\end{equation}
\end{corollary} 

\begin{proof} Assume first that $u_0$ satisfies \eqref{equd}. By Proposition \ref{seq}, 
every uniform limit $u$ to $u^\eps_\alpha$ satisfies \eqref{equd}. In particular $\overline{\{(x,t)\ |\ \nabla u(x,t)\neq 0\}}=\R^n$, 
so that by Theorem \ref{conv3} $i)$, we get \eqref{elpo}.  

If $u_0$ is convex and nonconstant then it is the supremum of all 
the linear functions $v_{a,b}(x)=a\cdot x+b$ such that $v_{a,b}\le u_0$ and $a\ne 0$.
Notice that, letting
$u_{a,b}$ be the a uniform limit of the solutions to \eqref{eq} with initial datum 
$v_{a,b}$ with $a\ne 0$, by the previous discussion we know that 
$u_{a,b}(x,t)= a\cdot x+b+t\int_0^1 g(s)ds$, for all $x\in\R^n$ and $t>0$.
As a consequence, by comparison principle we get 
\[
u(x,t)\ge \sup_{a,b:\,v_{a,b}\le u_0} u_{a,b}(x,t) =
\sup_{a,b:\,v_{a,b}\le u_0} (a\cdot x+b)+t\int_0^1 g(s)ds
= u_0(x)+t\int_0^1 g(s)ds.
\]
The opposite inequality follows from Corollary \ref{corollario0}.
\end{proof} 

\begin{proposition} \label{pro3}  
Let $u_0(x)=-C|x|$ with $C>0$ and 
let $\ue_\alpha$ be  the  solutions   to \eqref{eq}. Then $u^\eps_\alpha$ converges locally uniformly to the function
\begin{equation}\label{eqeq3}
v(x,t)  :=\min\left[ u_0(x)+t\int_0^1 g(s)ds, \ t\left(\int_0^1 \frac{1}{g(s)}ds\right)^{-1}\right].
\end{equation}
\end{proposition}

\begin{proof}
Letting $u$ be a limit of $u^\eps_\alpha$ given by Proposition \ref{seq}, we want to show that $u=v$.
By Corollary \ref{corollario0} we know that $u\le v$, so we are left to prove the opposite inequality.

First of all, we observe that, since $u_0$ is  radially  symmetric, then also  $u^\eps_\alpha(\cdot, t)$ is radially symmetric for every $t$, and then also $u(\cdot, t)$. 
So, we can write $u(x,t)=f(|x|,t)$, where $f(r,t):[0, +\infty)\times [0, +\infty)\to \R$ is a Lipschitz continuous function, with Lipschitz constant in $r$ less or equal to $C$. 
By Theorem \ref{conv3}, $f_t\in \big[\big(\int_0^1 \frac{1}{g(s)}ds\big)^{-1}, \ \int_0^1 g(s)ds\big]$ for a.e. $(r,t)$ and $f_t=\int_0^1 g(s)ds$ if $f_r\neq 0$. 
By Corollary \ref{corollario0} we also have that $f(0,t)=u(0,t)=t \big(\int_0^1 \frac{1}{g(s)}ds\big)^{-1}$ for every $t\geq 0$. 

Let $\tilde f(r,t):= \min_{r'\le r}f(r',t)$ be the largest nonincreasing function less or equal to $f$. Notice that $\tilde f(r,0)=-Cr$, and moreover  $\tilde f$ satisfies the same conditions as $f$, 
that is, there holds $\tilde f_r\in [-C,0]$,
$\tilde f_t\in \big[\big(\int_0^1 \frac{1}{g(s)}ds\big)^{-1}, \ \int_0^1 g(s)ds\big]$ for a.e. $(r,t)$ and $\tilde f_t=\int_0^1 g(s)ds$ if $\tilde f_r< 0$, and 
$\tilde f(0,t)=t \big(\int_0^1 \frac{1}{g(s)}ds\big)^{-1}$ for every $t\geq 0$. Since $\tilde f\le f$
it is enough to show that 
\begin{equation}\label{eqeq4}
\tilde f(r,t)\ge \min\left[ -Cr+t\int_0^1 g(s)ds, \ t\left(\int_0^1 \frac{1}{g(s)}ds\right)^{-1}\right]
\qquad \forall\, r,t\ge 0\,.
\end{equation}  
For $t\ge 0$ we let
$h(\cdot ,t):\big(-\infty, t \big(\int_0^1 \frac{1}{g(s)}ds\big)^{-1}\big]\to [0,+\infty)$ 
be the inverse of $\tilde f(\cdot,t)$, that is, $\tilde f(h(u,t),t)=u$ for a.e. $u\in (-\infty, t \big(\int_0^1 \frac{1}{g(s)}ds\big)^{-1}\big]$. In particular $h(u,0)=-\frac{u}{C}$. 
Then $h$ is nonincreasing in $u$,  $h_u(u,t)\le -1/C$ a.e., and $h_t(u,t)= -\big(\int_0^1 g(s)ds \big) h_u(u,t)\ge \big(\int_0^1 g(s)ds \big)/C$ for a.e. $(u,t)$. 

Let also $\tilde h(u,t):(-\infty, 0]\times [0, +\infty)\to [0,+\infty)$ be defined as
\[
\tilde h(u,t):= h\left(u+t \left(\int_0^1 \frac{1}{g(s)}ds\right)^{-1},t\right),
\]
so that there holds $\tilde h(u,0)=-u/C$,
$\tilde h_u(u,t)\le -1/C$ and 
$$
\tilde h_t(u,t)= \left[  \left(\int_0^1 \frac{1}{g(s)}ds\right)^{-1} - 
\int_0^1 g(s)ds \right] h_u(u,t)\ge \frac 1C
\left[  \int_0^1 g(s)ds - \left(\int_0^1 \frac{1}{g(s)}ds\right)^{-1}  \right] 
$$ 
for a.e. $(u,t)$. As a consequence, we get
\[
\tilde h(u,t)\ge -\frac uC + \frac tC \left[  \int_0^1 g(s)ds - \left(\int_0^1 \frac{1}{g(s)}ds\right)^{-1}  \right] 
\qquad \forall\, u,t\ge 0\,.
\] This, by definition,  reads $C h(u,t)\geq -u+ t\int_0^1 g(s)ds= -\tilde f(h(u,t),t) + t\int_0^1 g(s)ds$,
which is equivalent to \eqref{eqeq4}. This concludes the proof.
\end{proof}

\begin{remark}\label{remarkdati} \upshape  It is easy to check that the same conclusion of Proposition \ref{pro3} 
applies  to the case 
in which  $u_0(x)=\phi(|x-x_0|)$, for some $x_0\in\R^n$, with $\phi:[0, +\infty)\to [0, +\infty)$ is a  Lipschitz nonincreasing function. When $n=1$, it also applies to any Lipschitz initial datum such that 
$u_0$ is nondecreasing on $(-\infty,x_0]$ and nonincreasing on $[x_0,+\infty)$, for some $x_0\in\R$.
\end{remark} 

From Proposition \ref{pro3}, Corollary \ref{corollario0} and the comparison principle, we get the following convergence result.

\begin{corollary} \label{corollario3}  
Let $u_0(x)$ be a Lipschitz function bounded from above, let $\ue_\alpha$ be  the  solutions   to \eqref{eq}
and let $u$ be a limit of $\ue_\alpha$. Then   
$$
\lim_{t\to +\infty} \left[u(x,t)-\sup_{\R^n}u_0-t\left(\int_0^1 \frac{1}{g(s)}ds\right)^{-1}\right]=0
\qquad \text{locally uniformly.}$$
\end{corollary} 

\begin{proof}
Let $C=\|\nabla u_0\|_\infty$. Fix $\delta>0$ and choose $x_\delta\in\R^n$ such that $u_0(x_\delta)\geq \sup_{\R^n} u_0-\delta$. Up to a translation, we can assume $x_\delta=0$. Then $u_0(x)\geq \sup u_0-\delta-C|x|$, therefore by Proposition \ref{pro3} and by 
the comparison principle we get that \[u(x,t)\geq  
\min\left[ \sup_{\R^n} u_0-\delta-C|x|+t\int_0^1 g(s)ds, \ \sup_{\R^n} u_0-\delta+t\left(\int_0^1 \frac{1}{g(s)}ds\right)^{-1}\right].
\]
From this we deduce that for every compact set $K\subset\R^n$, there exists $t_K$ such that  
\begin{equation}\label{pro33} u(x,t)\geq   \sup_{\R^n} u_0-\delta+t\left(\int_0^1 \frac{1}{g(s)}ds\right)^{-1} \qquad \forall x\in K, \ t\geq t_K.
\end{equation}
Recall that by Corollary \ref{corollario0}, we have that $u(x,t)\leq \sup_{\R^n}u_0+t\left(\int_0^1 \frac{1}{g(s)}ds\right)^{-1}$. So, we conclude by \eqref{pro33} and 
the arbitrariness of $\delta>0$.
\end{proof}

\subsubsection{The one-dimensional case} 
In the one-dimensional case $n=1$, we provide a complete convergence result. 


We first introduce the following class of initial data, which we will denote by $\mathcal{L}$. 
A Lipschitz function $u_0$ belongs to $\mathcal{L}$ if  there exists
a sequence of points $\{x_i\}_{i\in I}$, with $I\subset\mathbb Z$, such that 
$x_i< x_{i+1}$ for all $i\in I$, the sequence $x_i$ han no accumulation points in $\R$,
and $u_0$ is monotone on all the segments of the form
$[x_i,x_{i+1}]$, with $u(x_i)\neq u(x_{i+1})$. 
Moreover, if $u_0$ is nonincreasing on $[x_i,x_{i+1}]$, then it is nondecreasing on $[x_{i+1},x_{i+2}]$, and viceversa. If $I$ is bounded from below (resp. from above), we also require that $u_0$ is monotone
on the half-line $(-\infty,\min I ]$ (resp. $[\max I,+\infty)$).



\smallskip

We start by proving the whole convergence for solutions $u^\eps_\alpha$ to \eqref{eq} with initial data belonging to $\mathcal{L}$.  Given $u_0\in\mathcal L$, we let $M\subset\R\cup \{\pm \infty\}$ be the set of points $x_i$ such that $u_0(x_i)>\max(u_0(x_{i-1)},u_0(x_{i+1}))$. 
If $I$ is bounded from below (resp. from above) and $u_0$ is decreasing on $(-\infty,\min I]$ 
(resp. increasing on $[\max I,+\infty)$), 
we also add $-\infty$ (resp. $+\infty$) to $M$, and we set 
$u_0(-\infty):=\sup_{(-\infty,\min I)}u_0$ (resp. $u_0(+\infty):=\sup_{(\max I,+\infty)}u_0$).

Setting for simplicity $x_{\min I-1}:=-\infty$ and $x_{\max I+1}:=+\infty$,
for $x\in M$ we let 
\[
I_x:= \begin{cases}
[x_{i-1},x_{i+1}] &\text{  if $x=x_i$ for some $i\in I$,}\\ 
(-\infty,x_{\min I}]&\text{  if $x=-\infty$,}\\ 
[x_{\max I},+\infty)&\text{  if $x=+\infty$.}
\end{cases}\]

For $x\in M$ we also let
\[
T_x:= \begin{cases}
\dfrac{u_0(x_i)-\max(u_0(x_{i-1)},u_0(x_{i+1}))}{\int_0^1 g(s)ds - \left(\int_0^1 \frac{1}{g(s)}ds\right)^{-1}}
&\text{  if $x=x_i$ for some $i\in I$,}\\ \\
\dfrac{u_0(-\infty)-u_0(x_{\min I})}{\int_0^1 g(s)ds - \left(\int_0^1 \frac{1}{g(s)}ds\right)^{-1}}
&\text{  if $x=-\infty$,}\\  \\
\dfrac{u_0(+\infty)-u_0(x_{\max I})}{\int_0^1 g(s)ds - \left(\int_0^1 \frac{1}{g(s)}ds\right)^{-1}}
&\text{  if $x=+\infty$.}
\end{cases}
\]
Notice that $T_x>0$ for all $x\in M$.

\begin{proposition}\label{ledzeppelin}
Let $n=1$ and let $u_0\in\mathcal{L}$.
Then $u^\eps_\alpha$ converges locally uniformly  to a function $u$ satisfying
\begin{equation}\label{heaven}
u(x,t)  =
\min\left[ u_{0}(x)+t\int_0^1 g(s)ds, \  u_0(\bar x)+ t\left(\int_0^1 \frac{1}{g(s)}ds\right)^{-1}\right]
\end{equation}
for $(x,t)\in I_{\bar x}\times [0,T_{\bar x}]$ and for all $\bar x\in M$.
\end{proposition}

\begin{proof}
Let $C$ be the Lipschitz constant of  $u_0$. For $\bar x\in M$ we let 
$u_{0,\bar x}$ be defined as  
\begin{eqnarray*}  
u_{0,\bar x}(x):=&\begin{cases}  u_0(x_{i-1})-Cx_{i-1}+Cx & \text { for $x\le x_{i-1}$}\\  u_0(x)&\text{  for $x\in [x_{i-1},x_{i+1}]$}\\  
u_0(x_{i+1})+Cx_{i+1}-Cx&\text{ for $x\ge x_{i+1}$}\end{cases}& \text{ if $\bar x= x_i$ for some $i$,}\\
u_{0,\bar x}(x):=& \begin{cases}    u_0(x)&\text{  for $x\in (-\infty, x_{\min I}]$}\\  
u_0(x_{\min I})+Cx_{\min I}-Cx&\text{ for $x\ge x_{\min I}$}\end{cases} 
&\text{ if $\bar x=-\infty$,}\\ 
u_{0,\bar x}(x):=&\begin{cases}  u_0(x_{\max I})-Cx_{\max I}+Cx & \text { for $x\le x_{\max I}$}\\  
u_0(x)&\text{  for $x\in [x_{\max I},+\infty)$}\end{cases}&\text{ if $\bar x=+\infty.$}
\end{eqnarray*} 

By construction,  $u_{0,\bar x}\leq u_0$ for every $\bar x\in M$, and  $u_0(x)= \sup_{\bar x\in M} u_{0,\bar x}(x)$.
Then by comparison we get that $u^\eps_\alpha\geq u^\eps_{\alpha, \bar x}$, where 
$u^\eps_{\alpha,\bar x}$ is the solution to
\eqref{eq} with initial datum $u_{0,\bar x}$. 
By Proposition \ref{pro3} and Remark \ref{remarkdati} we know that 
\[
\lim_{\eps\to 0} u^\eps_{\alpha,\bar x}(x,t)= \min\left[ u_{0,\bar x}(x)+t\int_0^1 g(s)ds, \ u_0(\bar x)+ t\left(\int_0^1 \frac{1}{g(s)}ds\right)^{-1}\right]\qquad \text{ locally uniformly.}
\]
Therefore, letting $u$ be a limit of $u^\eps_\alpha$, we conclude that  
\begin{equation}\label{maxmin} 
u(x,t)\geq \max_{\bar x\in M} \, u^\eps_{\alpha,\bar x} \qquad\forall (x,t)\in \R\times [0,+\infty).
\end{equation} 

On the other hand, reasoning as  in Corollary \ref{corollario0},
for all $\bar x\in M$ we also get
\begin{equation}\label{minmax} 
u(x,t)\leq \min\left[  u_{0,\bar x} (x)+t\int_0^1 g(s)ds, \,  u_0(\bar x)+ t\left(\int_0^1 \frac{1}{g(s)}ds\right)^{-1}\right] 
\end{equation}
for all $(x,t)\in I_{\bar x}\times [0,T_{\bar x}]$, which gives \eqref{heaven}.

\smallskip

Indeed, by Corollary \ref{corollario0} we know that 
\begin{equation}\label{minmoche}
u(x,t)\le u_0(x)+t \int_0^1 g(s)ds
= u_{0,\bar x}(x)+t \int_0^1 g(s)ds
\end{equation}
for all $(x,t)\in I_{\bar x}\times [0,T_{\bar x}]$.
By comparison principle, it follows that 
\begin{equation}\label{minmache}
u(x,t)\le u_0(\bar x)+t\left(\int_0^1 \frac{1}{g(s)}ds\right)^{-1}
\qquad \forall (x,t)\in I_{\bar x}\times [0,T_{\bar x}].
\end{equation} 
Inequality \eqref{minmax} then follows from \eqref{minmoche} and \eqref{minmache},
and the proof is concluded.
\end{proof}

Given $\delta>0$, we denote by $\mathcal L_\delta\subset\mathcal L$
the class of functions $u_0\in\mathcal L$ such that 
\begin{itemize}
\item[i)] $T_x\ge \delta$ for any $x\in M$;
\item[ii)] for any $x,y\in M$ either $T_x=T_y$ or $|T_x-T_y|\ge \delta$.
\end{itemize}
Notice that, for $u_0\in\mathcal L_\delta$, there exists an increasing sequence of 
times $T_i$, with $T_0=0$, such that $T_{i+1}\ge T_i + \delta$ and for any $x\in M$
there exist an index $i(x)$ such that $T_x=T_{i(x)}$.

\smallskip 

Proposition \ref{ledzeppelin} enables us to obtain the uniqueness of the limit function $u$
for initial data in $\mathcal L_\delta$.

\begin{corollary}\label{corldelta}
Let $n=1$ and let $u_0\in\mathcal{L}_\delta$ for some $\delta>0$.
Then $u^\eps_\alpha$ converges locally uniformly in $\R\times [0,+\infty)$ to a unique function $u$.
\end{corollary}

\begin{proof}
For $(x,t)\in \R\times [0,T_1]$ the result follows directly from Propositions \ref{seq} and \ref{ledzeppelin}. 
 Then, it is enough to observe that the function $u(x,T_1)$, given by the right-hand side
in \eqref{heaven}, still belongs to $\mathcal L_\delta$, possibly with a smaller set $M$.
Hence we can iteratively apply  Proposition \ref{ledzeppelin} on all the 
sets of the form $\R\times [T_i,T_{i+1}]$, and obtain the thesis.
\end{proof}

We now show an analogous result for general Lipschitz continuous initial data.

\begin{theorem}\label{teounicita} 
Let $n=1$,  let $u_0$ be a Lipschitz function,
and let $u^\eps_\alpha$ be the solution to \eqref{eq}. 
Then $u^\eps_\alpha$ converges locally uniformly to a unique function $u$.
\end{theorem}

\begin{proof} 
By Corollary \ref{corldelta} the result is true if $u_0\in \mathcal{L}_\delta$ for some $\delta>0$. 

Fix now a Lipschitz function  $u_0$, with Lipschitz constant $C>0$.
We observe that for any  $\delta>0$  it is possible to construct a function $u_\delta$, with Lipschitz constant $C$, such that $\|u_0-u_\delta\|_\infty\leq 2\delta$ and
$u_\delta\in \mathcal{L}_{\delta'}$ 
with $\delta'=\delta/(\int_0^1 g(s)ds - (\int_0^1 1/g(s)ds)^{-1})$. 
Indeed, fix $\delta>0$ and 
let $x_n= n \delta/C$, for $n\in \mathbb Z$. We then define $u_\delta$ as a piecewise linear function with slope $\pm C$ 
on each interval of the form $[x_{n-1}, x_{n}]$, satisfying $u_0(x_n)-\delta\leq u_\delta(x_n)\leq u_0(x_n)+\delta$ for every $n\in\mathbb Z$. This in turn implies that $u_\delta\in \mathcal{L}_{\delta'}$,
and $u_\delta(x)-2\delta\leq u_0(x)\leq u_\delta(x)+2\delta$ for every $x\in\R$.

Let now $\bar{u}^\eps_{ \alpha, \delta}$, $\underline{u}^\eps_{ \alpha, \delta}$ and $u^\eps_\alpha$ be the solutions to \eqref{eq} with initial data  
$u_\delta+2\delta$, $u_\delta-2\delta$ and $u_0$ respectively. By comparison principle we get that 
\[\underline{u}^\eps_{ \alpha, \delta}(x,t)\leq u^\eps_{ \alpha}(x,t)\leq \bar{u}^\eps_{ \alpha, \delta}(x,t)\qquad \forall (x,t)\in\R\times [0,+\infty).\]
By Corollary \ref{corldelta}, we knnow that there exist two functions $\underline{u}_\delta$ and $\bar u_\delta$ such that 
$$
\lim_{\eps\to 0} \underline{u}^\eps_{ \alpha, \delta}(x,t)=\underline{u}_\delta(x,t)\qquad
\text{and} \qquad
\lim_{\eps\to 0}\bar{u}^\eps_{ \alpha, \delta}(x,t)=\bar u_\delta(x,t)
\qquad \text{locally uniformly.}$$ 
 Moreover, by the explicit formula \eqref{heaven} we have that 
 $\bar u_\delta=\underline{u}_\delta+4\delta$.
Hence,  if $u$ is a locally uniform limit of $u^\eps_\alpha$
given by Proposition \ref{seq}, 
then it   satisfies \[\underline u_\delta(x,t)\leq u(x,t) \leq \bar{u}_\delta(x,t) = \underline u_\delta(x,t)+4\delta
\qquad \text{for every $(x,t)$ and every $\delta>0$. }\]
Letting $\delta\to 0$,
this implies that, if $u_1$ and $u_2$ are uniform limits of $u^\eps_\alpha$, then  $u_1= u_2$. 
\end{proof} 

\subsection{Case $\alpha=0$} 
The case $\alpha=0$ is completely analogous to the case $\alpha\in (0,1)$, the only difference is that in this case the limit problem \eqref{limit0} is of second order. 
The differential operator $\bar F$ appearing in the limit problem is defined as follows: \[  \bar F(p, X)= \begin{cases}  \tr X+\int_0^1 g(s)ds & p\neq 0\\  
\min\left(\tr X+\int_0^1 g(s)ds, \left(\int_0^1  g(s)^{-1}ds\right)^{-1}\right) & p=0.\end{cases}\] In particular $\bar F^\ast(p, X)= \tr X+\int_0^1 g(s)ds $ 
and  $\bar F_\ast(p, X)=\bar F(p, X)$ for every $(p, X)$.

According to the definition recalled in Section \ref{notation}, $u$ is a viscosity solution of the limit problem \eqref{limit0} if the following holds: 
if $\phi$ is  a smooth test function such that  $u(x_0,t_0)=\phi(x_0, t_0)$ and $u\leq \phi$, then  \[\phi_t(x_0, t_0)-\Delta \phi(x_0, t_0)\leq \int_0^1 g(s) ds.\] 
If, on the other hand, $u\geq \phi$, then    \[\phi_t(x_0, t_0)-\Delta \phi(x_0, t_0) \geq  \int_0^1 g(s) ds\qquad \text{if }\nabla \phi(x_0,t_0)\neq 0\] and  
\[\phi_t(x_0, t_0)\geq\min\left(\Delta \phi(x_0, t_0)+\int_0^1 g(s)ds,  \left( \int_0^1 g^{-1}(s) ds\right)^{-1}\right)\qquad \text{if }\nabla \phi(x_0,t_0)= 0.\]   

\begin{theorem} \label{conv4} Let $\ue_\alpha$ be  the  solution   to \eqref{eq} with $\alpha=0$ and with  $u_0\in C^{1,1}$. 
Every locally uniformly limit $u$ of $u^\eps_\alpha$ is a Lipschitz continuous function which satisfies  \eqref{estlip2} and solves in the viscosity sense the 
Cauchy problem \[\begin{cases} u_t-\bar F(\nabla u, \nabla^2 u)=0&\text{ in }\R^n\times (0, +\infty)\\ u(x,0)=u_0(x)& \text{ in }\R^n.\end{cases}\]
Moreover 
\begin{itemize} 
\item[i)]  $u$ is a viscosity solution to \[u_t-\Delta  u= \int_0^1g(s)ds\]  in every open set $\Omega\subset \R^n\times [0,+\infty)$, such that $\nabla u \neq 0$ a.e. in  $\Omega$.  
\item[ii)]  $u$ is a viscosity subsolution to \[u_t = \left(\int_0^1 \frac{1}{g(s)}ds\right)^{-1}\] at every point $(x_0,t_0)$ such that  $(0, X, \lambda)\in D^+ u(x_0, t_0)$ with $X< 0$ in the sense of matrices.
\item [iv)]  $u$ is a viscosity supersolution to \[u_t-\Delta  u= \int_0^1 g(s)ds\] at every point $(x_0,t_0)$ such that  $(0, X, \lambda)\in D^- u(x_0, t_0)$,  and there exist $\eta\in \R^n$, $\delta>0$,  such that $\eta^t X\eta>\delta|\eta|^2>0$.
\item[v)] If  $u_0$ is as in Corollario \ref{corollario1}, then $u$ is the solution to 
\begin{equation}\label{heat} \begin{cases}u_t-\Delta u=\int_0^1 g(s)ds \\
u(x,0)=u_0(x).\end{cases}\end{equation}
\end{itemize} 
\end{theorem} 
\begin{proof} The proofs are  completely analogous to those  of Theorem \ref{conv3} and Corollary \ref{corollario1}.
We just note that  
we need to use the fact that $\chi'_{p^\eps}\to1$ uniformly, as $\eps\to 0$.
  \end{proof}

 \section{Open problems} \label{secopen}
 We list some open problems which could be interesting to investigate in future works.
 \begin{enumerate} 
 \item 
An important question is the complete characterization of the limit function $u$ 
in the case $0\le \alpha<1$, and its uniqueness given the initial datum $u_0$.
At the moment, we are able to show uniqueness only in one-dimension, for $0<\alpha<1$. 
We recall that in general there is no uniqueness of viscosity solutions to 
the equations \eqref{limit01} and \eqref{limit0},
due to the discontinuity of the operators $\bar c_-$ and $\bar F$.
 
\item In this paper we only consider the case $g>0$. 
We expect that the same results are still valid in the case that $\int_0^1 g(s)ds>0$. 
We also expect that the cell problem could still be solved using an ODE argument, possibly more involved.  Note that if $\min g\le 0$, we get  $\bar c(0)=0$.

In the limiting case $\int_0^1 g(s)ds=0$, it is possible to solve the cell problem \eqref{cella}, with  $\bar c(p)\equiv 0$ for every $p$,
which means that the solutions $u^\eps_\alpha$ to \eqref{eq}  converge locally uniformly to the initial datum $u_0(x)$ for every $\alpha>0$, and to the solution of the heat equation with initial datum
$u_0$ for $\alpha=0$. 

\item In \cite{j} the long time rescaling \eqref{jerr} of \eqref{eq} has been considered, 
with $\int_0^1 g(s)ds=0$ and $\alpha\geq 1$.
In particular the author proved that $u^\eps_\alpha$ converge locally uniformly to a solution of a quasilinear parabolic equation, 
which in the case $\alpha>1$ coincides  the level set equation of the  mean curvature flow. 
For $\alpha=1$, in \cite[Thm 1.1]{j} it is proved that there exists a function $\theta:[0, +\infty)\to \R$ with $\theta(0)=1$ and $\theta(s)\in (0,1)$ for every $s>0$ 
such that the solutions to \eqref{jerr} with datum $u^\eps(x,0)=u_0(x)$  converge locally uniformly to the solution to
\begin{equation}\label{jerrlimit}
\begin{cases} 
u_t    - \tr\left(I-\theta(|\nabla u|^2)\frac{\nabla u}{|\nabla u|}\otimes\frac{\nabla u}{|\nabla u|}\right)\nabla^2 u =0&  \text{in }  \R^n\times (0, +\infty)
\\
u(x,0)=u_0(x) & \text{in }  \R^n.
\end{cases}\end{equation}
In \cite{m}, the $1$-dimensional case has been considered, for $\alpha=1$. In particular it is proved that $\lim_{|p|\to +\infty} \theta(|p|)=l>0$ (see \cite[Thm 1.1]{m}). 
The description of the long time rescaling in  the case $\int_0^1 g(s)ds=0$ and $\alpha<1$ is completely open. 

\item Another interesting issue is the case in which  the forcing term $g$ depends on both 
variables $x$ and $u$. 
In particular, if we assume that $g:\R^{n+1} \to \R$ is  Lipschitz continuous,  $\Z^{n+1}$-periodic, and strictly positive, then
 the homogenization problem reads as follows 
 \begin{equation}\label{levelgen2}  
u^\eps_t    =  \eps^\alpha \Delta \ue+ g\left(\frac{x}{\eps},\frac{\ue}{\eps } \right)\quad \text{in } (0, +\infty)\times \R^n\end{equation}    with initial data 
$\ue(0,x)  = u_0(x)$,   
where $u_0$ satisfies \eqref{id}. 

When $\alpha=1$, we obtain the following cell problem for \eqref{levelgen2}: 

\noindent for every $p\neq 0$ there exists a unique $\bar c(p)$ such that there exists a solution to 
\begin{equation}\label{cellnew} 
\begin{cases}  
-\Delta_y \chi -|p|^2 \chi_{zz} -2 p\cdot \nabla_{y}\chi_z +\bar c(p)\chi_z(y,z) -g(y, \chi(y,z)) =0   
& (y,z)\in [0,1]^{n+1}
\\ 
\chi(y,z+1)=\chi(y,z)+1 & 
\\ 
\chi(\cdot, z) \text{ periodic,}  & \end{cases} \end{equation} 

\noindent for $p=0$ there exists a unique
$\bar c(0)$ such that there exists a solution to 
\[\begin{cases}  - \Delta_y \chi+\bar c(0)\chi_z(y,z) -g(y, \chi(y,z)) =0   & 
(y,z)\in [0,1]^{n+1}
\\ \chi(y,z+1)=\chi(y,z)+1 & 
\\ \chi(\cdot ,z) \text{ periodic.}  &\end{cases} \]

Existence of traveling wave solutions for such problem and a  homogenization result for plane-like initial data have been given in \cite{cdn}. 
\end{enumerate}

\end{document}